\documentclass[a4paper,10pt,reqno]{amsart}
\usepackage{mathtools}
\usepackage{amssymb}
\usepackage{slashed}
\usepackage{hyperref}
\usepackage{enumitem}
\usepackage{array, booktabs}
\usepackage{multirow}
\hypersetup{
    colorlinks,
    citecolor=black,
    filecolor=black,
    linkcolor=black,
    urlcolor=black
}
\usepackage{orcidlink}

\newcolumntype{M}[1]{>{\centering\arraybackslash}m{#1}}
\theoremstyle{plain}
\newtheorem*{theorem*}{Theorem}
\newtheorem{theorem}{Theorem}[section]

\newtheorem{proposition}[theorem]{Proposition}
\newtheorem{lemma}[theorem]{Lemma}
\newtheorem{corollary}[theorem]{Corollary}
\theoremstyle{definition}
\newtheorem{question}[theorem]{Question}

\theoremstyle{remark}
\newtheorem{remark}[theorem]{Remark}

\newcommand{\dd}{\mathrm{d}}

\title{Left-invariant harmonic spinors on three-dimensional Lie groups}

\author[Alejandro Gil-García]{Alejandro Gil-García \orcidlink{0000-0002-9370-241X}}
\address{Mathematics Area, SISSA - Scuola Internazionale Superiore di Studi Avanzati,
via Bonomea 265, 34136 Trieste (TS), Italy}
\email{agilgarc@sissa.it}

\author[Giovanni Russo]{Giovanni Russo \orcidlink{0000-0002-0593-6719}}
\address{Mathematics Area, SISSA - Scuola Internazionale Superiore di Studi Avanzati,
via Bonomea 265, 34136 Trieste (TS), Italy}
\email{girusso@sissa.it}

\begin{document}

\begin{abstract}
We study the existence of left-invariant harmonic spinors on three-dimensional Lie groups equipped with a left-invariant pseudo-Riemannian metric. An existing formula for the spin Dirac operator acting on left-invariant spinors in the Riemannian setting is revised and specialised to our cases, in particular to almost Abelian Lie algebras. Focussing on dimension two and three, we find equivalent conditions for the Lie groups to admit left-invariant harmonic spinors in terms of constraints on the structure equations of the corresponding Lie algebras. We then identify those metrics (up to automorphism) carrying left-invariant harmonic spinors in each case.\bigskip

\medskip


\end{abstract}

\maketitle

\tableofcontents

\section*{Introduction}

The spin Dirac operator $\slashed{D}$ on a Riemannian spin manifold $(M,g)$ is a first order differential operator acting on a spinor bundle over $M$.
It is an elliptic and formally self-adjoint operator on the space of spinors with compact support with respect to an $L^2$ scalar product.
If $M$ is complete, $\slashed{D}$ is essentially self-adjoint on the completion of the said space of spinors with compact support.
If $M$ is compact, the spectrum of $\slashed{D}$ consists of discrete real eigenvalues with finite multiplicity \cite{bfgk,friedrich}.

Harmonic spinors are elements of the kernel of $\slashed{D}$.
They play a key role in particle physics \cite{witten} and are one of the many classes of spinors of interest in geometry (e.g.\ parallel, Killing spinors, etc.), see \cite{bfgk} and references therein.
In general, the dimension of $\ker\slashed{D}$ depends on the metric used to define the Dirac operator, and cannot be expressed in terms of topological invariants of the manifold \cite{hitchin}.
The Lichnerowicz--Weitzenb\"ock formula 
\[\slashed{D}^2 = \frac14 \mathrm{Scal}(g)+\Delta^S\]
relates the square of the Dirac operator with the scalar curvature $\mathrm{Scal}(g)$ of $(M,g)$ and the Bochner--Laplace operator $\Delta^S \coloneqq (\nabla^S)^*\circ \nabla^S$ (here $\nabla^S$ is the spin connection induced by the Levi-Civita connection on $(M,g)$, and $(\nabla^S)^*$ is its adjoint with respect to the above $L^2$ product).
A well-known implication of this formula when $M$ is compact is that if $\mathrm{Scal}(g) > 0$ then there are no non-trivial harmonic spinors.
The Lichnerowicz--Weitzenb\"ock formula holds for pseudo-Riemannian manifolds as well, with a slight variation of the definitions of $\slashed{D}$ and $\nabla^S$ in the Riemannian case (see Section \ref{sec:preliminaries-dirac} below).

Computing harmonic spinors amounts to solving a partial differential equation, which is challenging in general.
When a group symmetry is given, it is natural to restrict to the class of spinors that are invariant under the group action.
A natural question is thus to compute left-invariant harmonic spinors on Lie groups.

In this paper, we study the Dirac operator and left-invariant harmonic spinors on certain classes of pseudo-Riemannian real Lie groups, specifically in dimension three.
All Lie groups are parallelizable, and are thus spin with a unique, natural spin structure.
The Dirac operator acting on left-invariant spinors can be written in terms of a left-invariant coframe (see \cite[Prop.\ 5.2]{bazzoni-merchan-munoz} for the Riemannian case, cf.\ Proposition \ref{prop:formula_Dirac_operator} below for a revision of the previous result and a generalisation to any signature).
The existence of left-invariant harmonic spinors then naturally imposes restrictions on the structure constants of the Lie algebra of the Lie group.
Therefore, one can detect those left-invariant metrics (up to automorphism) carrying left-invariant harmonic spinors.
We study this process in dimension two and three, by grouping Lie algebras into various classes and using classification results of pseudo-Riemannian metrics \cite{ha-lee1,ha-lee3}.
To the best of our knowledge, this systematic approach has never appeared in the literature.
We now briefly summarise those classes of Lie algebras we are going to use.

Connected, simply connected three-dimensional Lie groups are in one-to-one correspondence with their Lie algebras, whose classification goes back to Bianchi \cite{bianchi, bianchi-essay}, cf.\ Milnor \cite{milnor} and our Tables \ref{table:three-dim-lie-algebras-uni}--\ref{table:three-dim-lie-algebras-non-uni}.
Besides the Abelian Lie algebra $\mathbb R^3$ and the simple ones $\mathfrak{su}(2)$ and $\mathfrak{sl}(2,\mathbb R)$, the remaining three-dimensional Lie algebras are almost Abelian, i.e.\ they are non-Abelian and admit a codimension one Abelian ideal.
Further, a Lie algebra is unimodular when the adjoint representation of any element is traceless.
One can then group three-dimensional Lie algebras into unimodular and non-unimodular ones.
Among the almost Abelian metric Lie algebras, one can further distinguish the class of isotropic ones, i.e.\ those on which the orthogonal complement of a codimension-one Abelian ideal is degenerate.
We will specialise our general formula for the Dirac operator to almost Abelian Lie groups in any dimension, distinguishing between the isotropic and the non-isotropic case.
This will be useful in particular to study harmonic spinors on the unique non-Abelian two-dimensional case.
In dimension three, we will see that a Riemannian Lie algebra admits left-invariant harmonic spinors only if it is unimodular.
As we will observe in Section \ref{subsec:four-dim-non-uni-ex}, this phenomenon does not generalise to dimension four.
In the Lorentzian case, we will discuss unimodular and non-unimodular cases separately.

A detailed study of the Dirac operator and its spectrum in the Riemannian case on the three-sphere $S^3 = \mathrm{SU}(2)$ was done by Hitchin \cite[Section 3.1]{hitchin}.
Here the standard bi-invariant metric has positive scalar curvature, so by the Lichnerowicz--Weitzenb\"ock formula this metric does not carry non-trivial harmonic spinors.
However, Hitchin showed that there are left-invariant Riemannian metrics on $S^3$ carrying harmonic spinors, but a consequence of his formulas is that these spinors cannot be left-invariant.
We provide a more general explanation for this, cf.\ Lemma \ref{lemma:cmatrix} below.
We will also see that there are left-invariant Lorentzian metrics on $S^3$ carrying left-invariant harmonic spinors.

The Dirac operator and its spectrum for left-invariant Riemannian metrics on the three-dimensional Heisenberg group (corresponding to the nilpotent Lie algebra $\mathfrak{heis}_3$) were studied by Ammann and B\"ar, see in particular \cite[Section 3, p.\ 230]{ammann-bar}.
Again, it is clear from their formulas that no left-invariant harmonic spinors exist. 
A conceptual explanation for this is provided again by our Lemma \ref{lemma:cmatrix}, where we show in particular that left-invariant harmonic spinors in the Riemannian non-Abelian case exist only if the Lie algebra is non-nilpotent.
Our present investigation was partly motivated by Ammann--B\"ar's and Hitchin's results.

Another natural problem could be to study the spectrum of the Dirac operator acting on left-invariant spinors on three-dimensional Lie groups.
In the Riemannian case and the unimodular Lorentzian case, we will see that left-invariant harmonic spinors exist if and only if the Dirac operator vanishes identically on left-invariant harmonic spinors.
This suggests the idea that finding harmonic spinors and non-harmonic eigenspinors are two inherently different problems in dimension three.
It is also worth mentioning that the techniques developed in \cite{harmful} could be used to study the problem of finding left-invariant harmonic spinors on, at least, unimodular Lie groups. These techniques offer an alternative approach to tackle this problem in higher dimensions. We leave these questions for future work.

The paper is structured as follows.
In Section \ref{sec:preliminaries} we review some basics of spin geometry and Lie algebras.
Further, we revise \cite[Prop.\ 5.2]{bazzoni-merchan-munoz} and generalise it to the pseudo-Riemannian case.
In Section \ref{sec:harmonic-spinors-almost-abelian-groups} we specialise the resulting general formula to the almost Abelian case.
We then see applications to the only non-Abelian two-dimensional Lie algebra.
In Sections \ref{sec:harmonic-spinors-riem-three-dim} and \ref{sec:harmonic-spinors-three-dimensional-lorentzian-groups} we detect all metrics carrying left-invariant harmonic spinors in dimension three.
These results are summarised in Tables \ref{table:Riemannian}--\ref{table:Lorentzian_nonuni}.
There the triple $\{x,y,z\}$ is a basis of the Lie algebra in the first column.
The corresponding Lie brackets are specified in the second column, whereas in the third column we have $3\times 3$ matrices representing metrics (up to equivalence) carrying left-invariant harmonic spinors with respect to the basis $\{x,y,z\}$.
In the fourth column we see relevant parametric conditions for each case, if any.
General classification results of Lorentzian metrics on unimodular three-dimensional Lie algebras are recalled in the Appendix and organised into Tables \ref{table:simple-algebras-lorentzian-metrics}--\ref{table:non-simple-algebras-lorentzian-metrics} at the end.

\textbf{Acknowledgments.} We thank Diego Conti, Anna Fino, and Lucía Martín-Merchán for useful discussions and comments. We also thank Diego Conti for suggesting a useful formula for the Ricci tensor of a left-invariant metric.
The first named author is supported by the International School for Advanced Studies (SISSA). 
The second named author is partially supported by INdAM--GNSAGA, and by the PRIN 2022 Project (2022K53E57) - PE1 - \lq\lq Optimal transport: new challenges across analysis and geometry\rq\rq.

\begin{table}
\centering
\caption{Three-dimensional Riemannian non-Abelian Lie algebras admitting left-invariant harmonic spinors}
\footnotesize
\begin{tabular}{p{25mm}p{25mm}p{45mm}M{17mm}} \toprule
\textbf{Lie algebra} & \textbf{Lie brackets} & \textbf{Metrics (up to equivalence)} & \textbf{Conditions} \\ \midrule
$\mathfrak{sl}(2,\mathbb{R})$ &
$\begin{aligned}
[x,y] &= 2z\\
[z,x] &= 2y\\
[z,y] &= 2x
\end{aligned}$
&
$\begin{pmatrix}
\lambda & 0 & 0\\
0 & \mu & 0\\
0 & 0 & \nu
\end{pmatrix}$
&
$\begin{gathered}
\lambda = \mu + \nu\\
\mu \geq \nu > 0
\end{gathered}$
\\
\cmidrule{1-4}
$\mathfrak{e}(1,1)$ &
$\begin{aligned}
[x,y] &= 0\\
[z,x] &= x\\
[z,y] &= -y
\end{aligned}$
&
$\begin{pmatrix}
1 & 0 & 0\\
0 & 1 & 0\\
0 & 0 & \nu
\end{pmatrix}$
&
$\nu > 0$
\\
\bottomrule
\end{tabular}
\label{table:Riemannian}
\end{table}

\begin{table}
\centering
\caption{Three-dimensional unimodular Lorentzian non-Abelian Lie algebras admitting left-invariant harmonic spinors}
\footnotesize
\begin{tabular}{p{25mm}p{25mm}p{45mm}M{17mm}} \toprule
\textbf{Lie algebra} & \textbf{Lie brackets} & \textbf{Metrics (up to equivalence)} & \textbf{Conditions} \\ \midrule
$\mathfrak{su}(2)$ &
$\begin{aligned}
[x,y] &= 2z \\
[z,x] &= 2y \\
[z,y] &= -2x
\end{aligned}$
&
$\begin{pmatrix}
\mu & 0 & 0\\
0 & \nu & 0\\
0 & 0 & -(\mu+\nu)
\end{pmatrix}$
&
$\mu\geq\nu>0$
\\
\cmidrule{1-4}
$\mathfrak{sl}(2,\mathbb R)$ &
$\begin{aligned}
[x,y] &= 2z\\
[z,x] &= 2y\\
[z,y] &= 2x
\end{aligned}$
&
$\begin{pmatrix}
\nu-\mu & 0 & 0\\
0 & -\mu & 0\\
0 & 0 & \nu
\end{pmatrix}$
&
$\nu>\mu>0$
\\
\cmidrule{3-4}
& &
$
\frac{4}{a^2\alpha N}\begin{pmatrix}
-N & 0 & \beta\\
0 & \frac{a^2\alpha}{N} & 0\\
\beta & 0 & \frac{\beta^2-\alpha^2}{N}
\end{pmatrix}
$ \newline
with $N=\sqrt{\alpha^2+\beta^2}$
&
$\begin{gathered}
2\beta^2=-a^2\alpha\\
-a^2<\alpha<0\\
\beta>0
\end{gathered}$
\\
\cmidrule{3-4}
& 
&
$\frac{1}{2ab}\begin{pmatrix}
a-8 & -a & 0\\
-a & a+8 & 0\\
0 & 0 & 8a/b
\end{pmatrix}$
&
$\begin{gathered}
a=-2b\\
b\neq0
\end{gathered}$
\\
\cmidrule{1-4}
$\mathfrak{e}(2)$ &
$\begin{aligned}
[x,y] &= z\\
[z,x] &= y\\
[z,y] &= 0
\end{aligned}$
&
$\begin{pmatrix}
0 & -1 & 0\\
-1 & u & 0\\
0 & 0 & v
\end{pmatrix}$
&
$u=-v<0$
\\
\cmidrule{1-4}
$\mathfrak{e}(1,1)$ &
$\begin{aligned}
[x,y] &= y\\
[z,x] &= z\\
[z,y] &= 0
\end{aligned}$
&
$\begin{pmatrix}
4/(u^2-v^2) & 0 & 0\\
0 & 1 & u/v\\
0 & u/v & 1
\end{pmatrix}$
&
$\begin{gathered}
u=0\\
v>0
\end{gathered}$
\\
\cmidrule{3-4}
& &
$\begin{pmatrix}
1/u & 0 & 0\\
0 & -1 & 0\\
0 & 0 & 1
\end{pmatrix}$
&
$u>0$
\\
\cmidrule{3-4}
& &
$\begin{pmatrix}
0 & 0 & 1\\
0 & 1 & 0\\
1 & 0 & 0
\end{pmatrix}$
&
--
\\
\cmidrule{1-4}
$\mathfrak{heis}_3$ &
$\begin{aligned}
[x,y] &= z\\
[z,x] &= 0\\
[z,y] &= 0
\end{aligned}$
&
$\begin{pmatrix}
1 & 0 & 0\\
0 & 0 & 1\\
0 & 1 & 0
\end{pmatrix}$
&
--
\\
\bottomrule
\end{tabular}
\label{table:Lorentzian_uni}
\end{table}

\begin{table}
\centering
\caption{Three-dimensional non-unimodular Lorentzian Lie algebras admitting left-invariant harmonic spinors}
\footnotesize
\begin{tabular}{p{17mm}p{27mm}p{41mm}M{27mm}} \toprule
\textbf{Lie algebra} & \textbf{Lie brackets} & \textbf{Metrics (up to equivalence)} & \textbf{Conditions} \\ \midrule
$\mathbb R^2\rtimes_{\mathrm{Id}}\mathbb R$ &
$\begin{aligned}
[x,y] &= 0\\
[z,x] &= x\\
[z,y] &= y
\end{aligned}$
&
$\begin{pmatrix}
1 & 0 & 0\\
0 & 0 & 1\\
0 & 1 & 0
\end{pmatrix}$
&
--
\\ \cmidrule{1-4}
$\mathfrak g(1)$ &
$\begin{aligned}
[x,y] &= 0\\
[z,x] &= x\\
[z,y] &= x+y
\end{aligned}$
&
$\begin{pmatrix}
\epsilon & 0 & 0\\
0 & -\epsilon/16 & 0\\
0 & 0 & \mu
\end{pmatrix}
$
&
$\begin{gathered}
\epsilon = \pm 1\\
\mu>0
\end{gathered}$
\\
\cmidrule{3-4}
&
\phantom{$\begin{aligned}
[x_1,x_2] &= 0\\
[x_3,x_1] &= x_1\\
[x_3,x_2] &= x_1 + x_2
\end{aligned}$}
&
$\begin{pmatrix}
0 & 0 & 1\\
0 & \mu & 0\\
1 & 0 & 0
\end{pmatrix}$
&
$\mu>0$
\\
\cmidrule{1-4}
$\mathfrak g(0)$ &
$\begin{aligned}
[x,y] &= 0\\
[z,x] &= 2x\\
[z,y] &= 0
\end{aligned}$
&
$\begin{pmatrix}
\epsilon & 1 & 0\\
1 & 3\epsilon/4 & 0\\
0 & 0 & \nu
\end{pmatrix}$
&
$\begin{gathered}
\epsilon = \pm 1\\
\nu>0
\end{gathered}$
\\
\cmidrule{3-4}
&
\phantom{$\begin{aligned}
[x_1,x_2] &= 0\\
[x_3,x_1] &= 2x_1\\
[x_3,x_2] &= 0
\end{aligned}$}
&
$\begin{pmatrix}
0 & 0 & 1\\
0 & 1 & 0\\
1 & 0 & 0
\end{pmatrix}, \quad 
\begin{pmatrix}
1 & 0 & 0\\
0 & 0 & 1\\
0 & 1 & 0
\end{pmatrix}$
&
--
\\
\cmidrule{1-4}
$\mathfrak g(-3)$ &
$\begin{aligned}
[x,y] &= 0\\
[z,x] &= 3x\\
[z,y] &= -y
\end{aligned}$
&
$\begin{pmatrix}
0 & 1 & 0\\
1 & \epsilon & 0\\
0 & 0 & \mu
\end{pmatrix}$
&
$\begin{gathered}
\epsilon \in \{0,\pm1\}\\
\mu>0
\end{gathered}$
\\
\cmidrule{3-4}
&
\phantom{$\begin{aligned}
[x_1,x_2] &= 0\\
[x_3,x_1] &= 3x_1\\
[x_3,x_2] &= -x_2
\end{aligned}$}
&
$\begin{pmatrix}
1 & 0 & 0\\
0 & 0 & 1\\
0 & 1 & 0
\end{pmatrix}$
&
--
\\
\cmidrule{1-4}
$\begin{gathered}
\mathfrak g(c)\\
c>1
\end{gathered}$ &
$\begin{aligned}
[x,y] &= 0\\
[z,x] &= y\\
[z,y] &= -cx+2y
\end{aligned}$
&
$\begin{pmatrix}
1 & 1 & 0\\
1 & \tau & 0\\
0 & 0 & \mu
\end{pmatrix}$
&
$\begin{gathered}
\tau=-c-6 \pm4\sqrt{c+3}\\
\mu>0
\end{gathered}$
\\
\cmidrule{1-4}
$\begin{gathered}
\mathfrak g(c)\\
c<1\\
c\notin\{-3,0\}
\end{gathered}$ &
$\begin{aligned}
[x,y] &= 0\\
[z,x] &= (1+\sqrt{1-c})x\\
[z,y] &= (1-\sqrt{1-c})y
\end{aligned}$ 
&
$\begin{pmatrix}
\epsilon & 1 & 0\\
1 & \epsilon(c+3)/4 & 0\\
0 & 0 & \mu
\end{pmatrix}$
&
$\begin{gathered}
\epsilon=\pm 1\\
\mu>0
\end{gathered}$
\\
\cmidrule{3-4}
&
\phantom{$\begin{aligned}
[x_1,x_2] &= 0\\
[x_3,x_1] &= (1+\sqrt{1-c})x_1\\
[x_3,x_2] &= (1-\sqrt{1-c})x_2
\end{aligned}$}
&
$\begin{pmatrix}
1 & 0 & 0\\
0 & 0 & 1\\
0 & 1 & 0
\end{pmatrix}$
&
--
\\
\bottomrule
\end{tabular}
\label{table:Lorentzian_nonuni}
\end{table}

\section{Preliminaries}
\label{sec:preliminaries}

\subsection{The Dirac operator on left-invariant spinors}
\label{sec:preliminaries-dirac}

Let $G$ be an $n$-dimensional connected, simply connected real Lie group equipped with a left-invariant pseudo-Riemannian metric $g$ of signature $(p,q)$. 
Every Lie group is parallelizable as it admits a global orthonormal left-invariant frame. 
Hence the orthonormal frame bundle over $G$ is just $G\times\mathrm{SO}(p,q)$ and its unique spin structure is $G\times\mathrm{Spin}(p,q)$. 

Let $\gamma\colon\mathrm{Cl}(p,q)\to \mathrm{End}_{\mathbb C}(\Sigma)$ be an irreducible complex representation of the Clifford algebra of $\mathbb R^n$ (where $\mathbb R^n$ models each tangent space of $G$ with scalar product $\langle{}\cdot{},{}\cdot{}\rangle $ of signature $(p,q)$). 
Here we define the Clifford algebra using the convention \[xy+yx=-2\langle x,y\rangle 1, \qquad x,y \in \mathbb R^n.\]
Then the spinor bundle of $G$ is the trivial complex vector bundle $S\coloneqq G\times\Sigma \to G$, and a spinor $\psi\in\Gamma(S)$ is identified with a map $\psi\colon G\to \Sigma$. 
We say that $\psi$ is \emph{left-invariant} if it is a constant map. 
We write Clifford multiplication of a vector field $x$ and a spinor $\psi \in \Sigma$ by $x\psi$ rather than $\gamma(x)\psi$.

Since all objects introduced are left-invariant, we will work at the Lie algebra level. Let $\mathfrak g$ be the Lie algebra of $G$ and let $g\colon\mathfrak g\times\mathfrak g\to\mathbb R$ be a pseudo-Riemannian metric of signature $(p,q)$ on $\mathfrak g$. Let $\{e_1,\ldots,e_n\}$ be an orthonormal basis of $(\mathfrak g,g)$ and denote by $\{e^1,\ldots,e^n\}\subset\mathfrak g^*$ its dual basis. The spin Dirac operator is a first order differential operator defined locally as \[\slashed{D}\colon \Gamma(S)\to \Gamma(S), \qquad \slashed{D}\psi = \sum_{k=1}^n \varepsilon_k e_k\nabla_{e_k}^S\psi,\] where $\nabla^S$ is the spin connection induced by the Levi-Civita connection $\nabla$ of $g$, the product with $e_k$ is given by Clifford multiplication, and $\varepsilon_k \coloneqq g(e_k,e_k) \in \{\pm 1\}$. Explicitly \[\nabla_{e_k}^S\psi = \partial_{e_k}\psi+\frac{1}{2}\sum_{1\leq i<j\leq n} \varepsilon_i\varepsilon_j g(\nabla_{e_k}e_i,e_j)e_ie_j\psi.\]

The signs $\varepsilon_k$ make the above formulas invariant under the pseudo-orthogonal group. We recall that elements in the kernel of $\slashed{D}$ are called \emph{harmonic} spinors.

We are interested in the action of $\slashed{D}$ on left-invariant spinors. Since these are constant maps on $G$ with values in $\Sigma$, they can be identified with elements of the complex representation space $\Sigma$. If $\psi$ is a left-invariant spinor, clearly $\partial_{e_k}\psi=0$ for all $k$.
An explicit formula for the Dirac operator acting on left-invariant spinors in terms of the differential forms $\dd e^i$ was obtained in \cite[Prop.\ 5.2]{bazzoni-merchan-munoz}. There the authors consider the case where $g$ is a Riemannian metric. However, we noticed that there is a missing factor $2$ in their formula. Hereafter we revise the proof of \cite[Prop.\ 5.2]{bazzoni-merchan-munoz}, and include the more general case where $g$ is an indefinite metric. 
If $\alpha $ is any form and $x$ a vector, we write $\iota_x\alpha $ for the interior product $\alpha(x,{}\cdot{})$.

\begin{proposition}
\label{prop:formula_Dirac_operator}
Let $(\mathfrak g,g)$ be a pseudo-Riemannian Lie algebra, $\{e_1,\ldots,e_n\}$ be an orthonormal basis of $(\mathfrak g,g)$, and $\psi\in\Sigma$ be a left-invariant spinor. 
Then 
\[\slashed{D}\psi=-\frac{1}{4}\sum_{i=1}^n(\varepsilon_ie^i\wedge\dd e^i+2\iota_{e_i}\dd e^i)\psi,\]
where $\varepsilon_i= g(e_i,e_i)\in\{\pm1\}$.
\end{proposition}

\begin{proof}
The Koszul formula on $(\mathfrak g,g)$ reads 
\[2g(\nabla_xy,z)=g([x,y],z)-g([y,z],x)+g([z,x],y)\] 
for all $x,y,z\in\mathfrak g$, where $\nabla$ denotes the Levi-Civita connection of $g$. 
In particular, we have 
\begin{align*}
2g(\nabla_{e_i}e_j,e_k)&=g([e_i,e_j],e_k)-g([e_j,e_k],e_i)+g([e_k,e_i],e_j)\\
&=\varepsilon_ke^k([e_i,e_j])-\varepsilon_ie^i([e_j,e_k])+\varepsilon_je^j([e_k,e_i])\\
&=-\varepsilon_k\dd e^k(e_i,e_j)+\varepsilon_i\dd e^i(e_j,e_k)-\varepsilon_j\dd e^j(e_k,e_i).
\end{align*}
The formula for the spin connection $\nabla^S$ acting on a left-invariant spinor $\psi\in\Sigma$ is given by 
\[2\nabla^S_{e_i}\psi=\sum_{j<k}\varepsilon_j\varepsilon_kg(\nabla_{e_i}e_j,e_k)e_je_k\psi.\]
Hence, we obtain 
\begin{align*}
4\nabla^S_{e_i}\psi&=\sum_{j<k}\varepsilon_j\varepsilon_k(\varepsilon_i\dd e^i(e_j,e_k)-\varepsilon_j\dd e^j(e_k,e_i)-\varepsilon_k\dd e^k(e_i,e_j))e_je_k\psi\\
&=\varepsilon_i\dd e^i\psi-\sum_{j<k}\varepsilon_k\dd e^j(e_k,e_i)e_je_k\psi-\sum_{j<k}\varepsilon_j\dd e^k(e_i,e_j)e_je_k\psi,
\end{align*} 
where we have used that a two-form $\omega$ acts on $\psi$ as $\omega\psi = \sum_{j<k} \varepsilon_j\varepsilon_k \omega(e_j,e_k)e_je_k\psi$. 
Note the following:
\begin{align*}
\sum_{j,k}\varepsilon_j\dd e^k(e_i,e_j)e_je_k&=\sum_{j<k}\varepsilon_j\dd e^k(e_i,e_j)e_je_k+\sum_{j>k}\varepsilon_j\dd e^k(e_i,e_j)e_je_k \\
& \qquad +\sum_{k}\varepsilon_k\dd e^k(e_i,e_k)e_k^2.
\end{align*}
Swapping the indices $j$ and $k$ in the second summand we obtain 
\[\sum_{j>k}\varepsilon_j\dd e^k(e_i,e_j)e_je_k=\sum_{k>j}\varepsilon_k\dd e^j(e_i,e_k)e_ke_j=\sum_{j<k}\varepsilon_k\dd e^j(e_k,e_i)e_je_k.\]
Using $e_k^2=-\varepsilon_k$, the last summand becomes 
\[\sum_k\varepsilon_k\dd e^k(e_i,e_k)e_k^2=\sum_k\dd e^k(e_k,e_i).\]
Therefore 
\[4\nabla^S_{e_i}\psi=\varepsilon_i\dd e^i\psi-\sum_{j,k}\varepsilon_j\dd e^k(e_i,e_j)e_je_k\psi+\sum_k\dd e^k(e_k,e_i)\psi.\]
We now compute the Dirac operator acting on $\psi\in\Sigma$: 
\begin{align*}
4\slashed{D}\psi&=4\sum_i\varepsilon_ie_i\nabla^S_{e_i}\psi\\
&=\sum_ie_i\dd e^i\psi-\sum_{i,j,k}\varepsilon_i\varepsilon_j\dd e^k(e_i,e_j)e_ie_je_k\psi+\sum_{i,k}\varepsilon_i\dd e^k(e_k,e_i)e_i\psi.
\end{align*}
For the second summand we have 
\[\sum_{i,j,k}\varepsilon_i\varepsilon_j\dd e^k(e_i,e_j)e_ie_je_k=\sum_{i<j,k}\varepsilon_i\varepsilon_j\dd e^k(e_i,e_j)e_ie_je_k+\sum_{i>j,k}\varepsilon_i\varepsilon_j\dd e^k(e_i,e_j)e_ie_je_k,\] 
and relabeling $i$ and $j$ in the second term we obtain 
\[\sum_{i,j,k}\varepsilon_i\varepsilon_j\dd e^k(e_i,e_j)e_ie_je_k\psi=2\sum_{i<j,k}\varepsilon_i\varepsilon_j\dd e^k(e_i,e_j)e_ie_je_k\psi=2\sum_k\dd e^ke_k\psi.\]
The third summand in the expression of $4\slashed{D}\psi$ is just 
\[\sum_{i,k}\varepsilon_i\dd e^k(e_k,e_i)e_i\psi=\sum_k\iota_{e_k}\dd e^k\psi,\] 
where we have used that a one-form $\alpha$ acts on $\psi$ as $\alpha\psi = \sum_i \varepsilon_i \alpha(e_i)e_i\psi$. 
Changing the index $k$ by $i$ we obtain 
\[4\slashed{D}\psi=\sum_i(e_i\dd e^i-2\dd e^ie_i+\iota_{e_i}\dd e^i)\psi.\]
Finally, using $e_i\dd e^i\psi=(\varepsilon_ie^i\wedge\dd e^i-\iota_{e_i}\dd e^i)\psi$ and $\dd e^ie_i\psi=(\varepsilon_i\dd e^i\wedge e^i+\iota_{e_i}\dd e^i)\psi$ we obtain the claim.
\end{proof}
We will be interested in left-invariant harmonic spinors on $G$, and in those left-invariant metrics carrying them.

\begin{remark}
By Proposition \ref{prop:formula_Dirac_operator}, all left-invariant spinors on Abelian Lie algebras are harmonic, regardless of the choice of a metric. 
We will then assume our Lie algebras to be non-Abelian.
\end{remark}

\begin{remark}
Proposition \ref{prop:formula_Dirac_operator} implies that the existence of left-invariant harmonic spinors on $G$ depends on the structure equations of the Lie algebra $\mathfrak g$. 
Because of this, we say that the Lie algebra $\mathfrak g$ admits left-invariant harmonic spinors, although these are sections of the spinor bundle over $G$. 
\end{remark}

\subsection{Three-dimensional Lie algebras}

We will focus on three-dimensional connected, simply connected real Lie groups. These are in one-to-one correspondence with their Lie algebras. The classification of three-dimensional real Lie algebras goes back to Bianchi \cite{bianchi,bianchi-essay}, see also Milnor \cite{milnor}.

Let us recall some facts about Lie algebras. 
A Lie algebra $\mathfrak g$ is called \emph{unimodular} if $\mathrm{tr}(\mathrm{ad}(x))=0$ for all $x\in\mathfrak g$, where $\mathrm{ad}\colon\mathfrak g\to\mathrm{Der}(\mathfrak g)$, $\mathrm{ad}(x) \coloneqq [x,{}\cdot{}]$, is the adjoint representation of $\mathfrak g$. 
The \emph{Killing form} of a Lie algebra $\mathfrak g$ is the symmetric bilinear form $B_\mathfrak g\colon\mathfrak g\times\mathfrak g\to\mathbb R$ defined by \[B_\mathfrak g(x,y)\coloneqq\mathrm{tr}(\mathrm{ad}(x)\circ\mathrm{ad}(y)), \qquad x,y \in \mathfrak g.\] The Killing form $B_\mathfrak g$ is invariant under automorphisms of the Lie algebra $\mathfrak g$, so it does not depend on the choice of a basis for $\mathfrak g$. We denote by $\sigma(\mathfrak g)\coloneqq(p,q,r)$ the \emph{signature} of the Killing form $B_\mathfrak g$ of $\mathfrak g$, where $p$ is the number of positive directions, $q$ the number of negative ones, $r$ the number of null ones, and $p+q+r=\dim_{\mathbb R}(\mathfrak g)$.
The list of non-Abelian three-dimensional Lie algebras is presented in Tables~\ref{table:three-dim-lie-algebras-uni}--\ref{table:three-dim-lie-algebras-non-uni}.

\begin{table}
\centering
\caption{Three-dimensional unimodular non-Abelian Lie algebras}
\footnotesize
\begin{tabular}{p{30mm}p{52mm}M{30mm}} \toprule
\textbf{Lie algebra} $\mathfrak g$ & \textbf{Description} &  \textbf{Signature} $\sigma(\mathfrak g)$ \\ \midrule
$\mathfrak{su}(2)\cong \mathfrak{so}(3)$ & Compact real form of $\mathfrak{sl}(2,\mathbb C)$, simple & $(0,3,0)$ \\ \midrule
$\mathfrak{sl}(2,\mathbb R)\cong \mathfrak{so}(2,1)$ & Split real form of $\mathfrak{sl}(2,\mathbb C)$, simple & $(2,1,0)$ \\ \midrule
$\mathfrak{e}(2) = \mathfrak{r}_{3,0}'$ & Lie algebra of the group of rigid motions of $\mathbb R^2$, solvable & $(0,1,2)$ \\ \midrule
$\mathfrak{e}(1,1) = \mathfrak{r}_{3,-1}=\mathfrak{sol}_3$ & Lie algebra of the group of rigid motions of $\mathbb R^{1,1}$, completely solvable & $(1,0,2)$ \\ \midrule
$\mathfrak{heis}_3$ & Lie algebra of the Heisenberg group, nilpotent & $(0,0,3)$ \\ 
\bottomrule
\end{tabular}
\label{table:three-dim-lie-algebras-uni}
\end{table}

\begin{table}
\centering
\caption{Three-dimensional non-unimodular Lie algebras}
\footnotesize
\begin{tabular}{p{30mm}p{37mm}M{45mm}} \toprule
\textbf{Lie algebra} $\mathfrak g$ & \textbf{Description} & \textbf{Signature} $\sigma(\mathfrak g)$ \\ \midrule
$\mathbb R^2 \rtimes_{\mathrm{Id}} \mathbb R$ & Lie algebra of the real hyperbolic space $\mathbb{R}\mathrm{H}^3$ as a solvable Lie group & $(1,0,2)$ \\ \midrule
$\mathfrak g(c)\coloneqq\mathbb R^2\rtimes_{D(c)}\mathbb R$ & $D(c)$ conjugate to $\left(\begin{smallmatrix}0 & -c \\ 1 & 2\end{smallmatrix}\right)$ (up to scaling), solvable & $\begin{aligned}
        (0,1,2),& \quad c>2,\\
        (0,0,3),& \quad c=2,\\
        (1,0,2),& \quad c<2.\\
    \end{aligned}$ \\ 
\bottomrule
\end{tabular}
\label{table:three-dim-lie-algebras-non-uni}
\end{table}
\begin{remark}
Comparing to the Bianchi classification: $\mathbb R^2\rtimes_{\mathrm{Id}}\mathbb R$ corresponds to type V; 
$\mathfrak g(0)$ corresponds to type III; 
$\mathfrak g(1)$ corresponds to type IV; 
$\mathfrak g(c)$ with $c<1$ and $c\neq0$ corresponds to type VI; 
finally $\mathfrak g(c)$ with $c>1$ corresponds to type VII.
\end{remark}

Let $\mathfrak g$ be a Lie algebra. Two metrics $g_1$ and $g_2$ on $\mathfrak g$ are said to be \emph{equivalent} (up to automorphism) if there exists an automorphism $A\in\mathrm{Aut}(\mathfrak g)$ such that $g_2(x,y)=g_1(Ax,Ay)$ for all $x,y\in\mathfrak g$. In terms of matrices, this equation reads $g_2=A^Tg_1A$. Riemannian and Lorentzian metrics on three-dimensional Lie algebras have been classified up to equivalence in \cite{ha-lee1} and \cite{boucetta-chakkar,ha-lee3}, respectively. We will compare our classifications with these results to explicitly identify which metrics carry left-invariant harmonic spinors.

\section{Harmonic spinors on almost Abelian Lie groups}
\label{sec:harmonic-spinors-almost-abelian-groups}

We consider the class of almost Abelian Lie algebras and obtain a general formula for the Dirac operator acting on left-invariant spinors in any dimension and signature. We then apply this formula to find left-invariant harmonic spinors on the unique non-Abelian two-dimensional Lie algebra, denoted $\mathfrak{aff}(\mathbb R)$ (the Lie algebra of the group of affine transformations of $\mathbb R$), which is in fact almost Abelian. We will show that on $\mathfrak{aff}(\mathbb R)$ left-invariant harmonic spinors exist only when a Lorentzian metric is considered, and that such a metric is unique up to equivalence. In dimension three, every Lie algebra except the simple ones, namely $\mathfrak{su}(2)$ and $\mathfrak{sl}(2,\mathbb R)$, is almost Abelian. This will be useful to classify three-dimensional Lie algebras admitting left-invariant harmonic spinors in the next sections.

\subsection{General formulas}
\label{subsec:gen-formulas}
Let $\mathfrak g$ be a real Lie algebra of dimension $n$. We say that $\mathfrak g$ is \emph{almost Abelian} if it is non-Abelian and admits a codimension one Abelian ideal $\mathfrak h$. In particular, the derived subalgebra $[\mathfrak g,\mathfrak g]\subset\mathfrak h$ is Abelian and $\mathfrak g$ is two-step solvable. We can choose a basis $\{v_1,\ldots,v_n\}$ of $\mathfrak g$ such that 
\[\mathfrak h=\mathrm{Span}_{\mathbb R}\{v_1,\ldots,v_{n-1}\}\cong \mathbb R^{n-1},\qquad\mathrm{ad}(v_n)\mathfrak h\subset\mathfrak h.\]
The whole Lie algebra structure of $\mathfrak g$ is determined by the derivation 
\[D\coloneqq\mathrm{ad}(v_n)|_{\mathfrak h}\in\mathrm{Der}(\mathfrak h)=\mathrm{End}_{\mathbb R}(\mathbb R^{n-1}),\] 
allowing us to identify $\mathfrak g$ with the semidirect product $\mathfrak h\rtimes_D\mathbb R$. More precisely, if $D=(d_{ij})$, then the Lie brackets of $\mathfrak g$ are as follows: 
\[[v_i,v_j]=0,\qquad [v_n,v_i]=Dv_i\eqqcolon\sum_{k=1}^{n-1}d_{ki}v_k\] 
for $i,j=1,\ldots,n-1$. 
Equivalently, the differentials of the elements of the dual basis $\{v^1,\ldots,v^n\}$ of $\mathfrak g^*$ are 
\begin{equation}
\label{eq:differentials_almost_Abelian}
\dd v^i=-(Dv^i)\wedge v^n=\sum_{j=1}^{n-1}d_{ij}v^j\wedge v^n,\qquad \dd v^n=0
\end{equation} for $i=1,\ldots,n-1$, where $D$ acts on a one-form $\alpha$ by $D\alpha\coloneqq-\alpha\circ D$.

Let $g$ be a pseudo-Riemannian metric on $\mathfrak g=\mathfrak h\rtimes_D\mathbb R$. 
The one-dimensional orthogonal complement $\mathfrak h^\perp$ with respect to $g$ can be either non-degenerate or null. 
Following the terminology of \cite{conti-gilgarcia}, 
in the first case we say that $(\mathfrak g=\mathfrak h\rtimes_D\mathbb R,g)$ is \emph{non-isotropic}, whereas in the second case we say that $(\mathfrak g=\mathfrak h\rtimes_D\mathbb R,g)$ is \emph{isotropic}.
\begin{remark}
\label{rmk:heis}
The almost Abelian Lie algebra $\mathfrak{heis}_3$ is the only three-dimensional almost Abelian one where the codimension one ideal $\mathfrak h$ is not unique \cite[Prop.\ 1]{freibert}, and depending on its choice $\mathfrak h^{\perp}$ may be degenerate or not.
In our general statements, such $\mathfrak h$ will always be fixed so as to avoid ambiguity.
\end{remark}
Suppose that $(\mathfrak g=\mathfrak h\rtimes_D\mathbb R,g)$ is non-isotropic. Then $\mathfrak h^\perp$ is spanned by a vector $e_n$ which we assume satisfies $\varepsilon_n=g(e_n,e_n)\in\{\pm1\}$. Take an orthonormal basis $\{e_1,\ldots,e_{n-1}\}$ of $\mathfrak h$ with respect to $g|_{\mathfrak h\times\mathfrak h}$. Then $\{e_1,\ldots,e_n\}$ is an orthonormal basis of $(\mathfrak g,g)$. We assume that $\mathfrak h=\mathrm{Span}_{\mathbb R}\{e_1,\ldots,e_{n-1}\}$ and $D=\mathrm{ad}(e_n)|_\mathfrak h$. In this situation we obtain the following formula for the Dirac operator.
\begin{proposition}
\label{prop:Dirac_almost_Abelian}
Let $(\mathfrak g=\mathfrak h\rtimes_D\mathbb R,g)$ be a non-isotropic almost Abelian pseudo-Riemannian Lie algebra and $\psi\in\Sigma$ be a left-invariant spinor. 
Then 
\[\slashed{D}\psi=-\frac{1}{4}\sum_{1\leq i<j\leq n-1}(\varepsilon_id_{ij}-\varepsilon_jd_{ji})(e^i\wedge e^j\wedge e^n)\psi-\frac{1}{2}\mathrm{tr}(D)e^n\psi,\] 
where $D=(d_{ij})$ and $\varepsilon_i=g(e_i,e_i)\in\{\pm1\}$.
\end{proposition}
\begin{proof}
The general formula for the Dirac operator is given by Proposition \ref{prop:formula_Dirac_operator}. Using \eqref{eq:differentials_almost_Abelian} we obtain 
\begin{align*}
-4\slashed{D}\psi&=\sum_{i=1}^{n-1}(\varepsilon_ie^i\wedge\dd e^i+2\iota_{e_i}\dd e^i)\psi+(\varepsilon_ne^n\wedge\dd e^n+2\iota_{e_n}\dd e^n)\psi \\
&=-\sum_{i=1}^{n-1}(\varepsilon_ie^i\wedge De^i\wedge e^n+2\iota_{e_i}(De^i)e^n)\psi.
\end{align*}
Since $De^i=-\sum_{j=1}^{n-1}d_{ij}e^j$ and $\iota_{e_i}(De^i)=-\sum_{j=1}^{n-1}d_{ij}\iota_{e_i}(e^j)=-d_{ii}$, we have 
\[\sum_{i=1}^{n-1}\iota_{e_i}(De^i)e^n\psi=-\mathrm{tr}(D)e^n\psi.\]
Now we compute 
\begin{align*}
\sum_{i=1}^{n-1}\varepsilon_ie^i\wedge De^i&=-\sum_{i,j=1}^{n-1}\varepsilon_id_{ij}e^i\wedge e^j \\
&=-\sum_{1\leq i<j\leq n-1}\varepsilon_id_{ij}e^i\wedge e^j-\sum_{1\leq j<i\leq n-1}\varepsilon_id_{ij}e^i\wedge e^j\\
&=-\sum_{1\leq i<j\leq n-1}(\varepsilon_id_{ij}-\varepsilon_jd_{ji})e^i\wedge e^j,
\end{align*} 
and we are done.
\end{proof}
Given a pseudo-Riemannian Lie algebra $(\mathfrak g,g)$ and an endomorphism $f\colon\mathfrak g\to\mathfrak g$, we denote by $f^*$ the adjoint endomorphism of $f$ with respect to the metric $g$, that is $g(f^*{}\cdot{},{}\cdot{})\coloneqq g({}\cdot{},f{}\cdot{})$.

Given an arbitrary endomorphism $D\in\mathrm{End}_{\mathbb R}(\mathbb R^{n-1})$, it is not difficult to check that 
\[\widehat{D}\coloneqq\frac{1}{2}(D+D^*)-\frac{1}{n-1}\mathrm{tr}(D)\mathrm{Id}_{n-1}\] 
satisfies $\varepsilon_i\widehat{d}_{ij}-\varepsilon_j\widehat{d}_{ji}=0$ for all $1\leq i<j\leq n-1$ and $\mathrm{tr}(\widehat{D})=0$.
\begin{corollary}
Let $(\widehat{\mathfrak g}=\mathfrak h\rtimes_{\widehat{D}}\mathbb R,g)$ be a non-isotropic almost Abelian pseudo-Riemannian Lie algebra. Then all left-invariant spinors on it are harmonic.
\end{corollary}
In the Riemannian setting, we obtain the following obstruction for an almost Abelian Lie algebra to admit left-invariant harmonic spinors.
\begin{corollary}
\label{cor:almost_Abelian_non_unimodular}
Let $(\mathfrak g=\mathfrak h\rtimes_D\mathbb R,g)$ be a non-unimodular almost Abelian Riemannian Lie algebra. Then it does not admit left-invariant harmonic spinors.
\end{corollary}
\begin{proof}
An almost Abelian Lie algebra is unimodular if and only if $\mathrm{tr}(D)=0$. 
Now note that we can rewrite the formula in Proposition \ref{prop:Dirac_almost_Abelian} as 
\[\slashed{D}\psi=-\frac{1}{2}e^n(\omega+\mathrm{tr}(D)\mathrm{Id}_{\Sigma})\psi,\] 
with 
\[\omega=\frac{1}{2}\sum_{1\leq i<j\leq n-1}(d_{ij}-d_{ji})e^i\wedge e^j.\]
Since the endomorphism associated to the basis vector $e^n$ is invertible, we have that $\slashed{D}\psi=0$ if and only if $(\omega+\mathrm{tr}(D)\mathrm{Id}_{\Sigma})\psi=0$. 
Take the standard Hermitian inner product $\langle{}\cdot{},{}\cdot{}\rangle$ on $\Sigma$, which satisfies 
\[\langle\alpha\psi_1,\psi_2\rangle+\langle\psi_1,\alpha \psi_2\rangle=0\] 
for all $\alpha\in\mathfrak g^*$ and all $\psi_1,\psi_2\in\Sigma$. 
Then $\langle\omega\psi,\psi\rangle=-\langle\psi,\omega\psi\rangle$. 
Now we compute the square norm of $\omega\psi+\mathrm{tr}(D)\psi$ for a non-zero $\psi\in\Sigma$: 
\[\lVert\omega\psi+\mathrm{tr}(D)\psi\rVert^2=\lVert\omega\psi\rVert^2+\mathrm{tr}(D)(\langle\omega\psi,\psi\rangle+\langle\psi,\omega\psi\rangle)+\mathrm{tr}(D)^2\lVert\psi\rVert^2.\]
The term in the middle vanishes and $\mathrm{tr}(D)\neq0$ since $\mathfrak g$ is non-unimodular, whence
\[\lVert\omega\psi+\mathrm{tr}(D)\psi\rVert^2=\lVert\omega\psi\rVert^2+\mathrm{tr}(D)^2\lVert\psi\rVert^2\geq\mathrm{tr}(D)^2\lVert\psi\rVert^2>0,\] 
and there are no left-invariant harmonic spinors.
\end{proof}

Suppose that $(\mathfrak g=\mathfrak h\rtimes_D\mathbb R,g)$ is isotropic. 
Then $\mathfrak h^\perp\subset\mathfrak h$ and the metric $g$ on $\mathfrak h$ is degenerate. 
Indeed, if $\mathfrak h^\perp\not\subset\mathfrak h$, then $\mathfrak h^\perp$ would be orthogonal to $\mathfrak g=\mathfrak h\oplus\mathfrak h^\perp$, but this is impossible as $g$ is non-degenerate. 
Fix a generator $v_{n-1}$ of $\mathfrak h^\perp$ and take a complement $\mathfrak{v}$ of $\mathrm{Span}_{\mathbb R}\{v_{n-1}\}$ inside $\mathfrak h$. 
Then the complement $\mathfrak v$ is non-degenerate with respect to $g|_{\mathfrak v\times\mathfrak v}$ and we can choose an orthonormal basis $\{v_1,\dots,v_{n-2}\}$ of $(\mathfrak v,g|_{\mathfrak v \times\mathfrak v})$. 
We now choose a vector $v_n$ orthogonal to $\mathfrak{v}$ such that $g(v_n,v_n)=0$ and $g(v_{n-1},v_n)=1$. 
The metric $g$ in the basis $\{v_1,\dots,v_{n-2},v_{n-1},v_n\}$ is 
\[g=\sum_{i=1}^{n-2}\varepsilon_iv^i\otimes v^i+v^{n-1}\odot v^n,\] 
where $v^i\odot v^j\coloneqq v^i\otimes v^j+v^j\otimes v^i$. 
An orthonormal basis of $(\mathfrak g,g)$ is given by 
\[e_i=v_i,\qquad e_{n-1}=\frac{\sqrt{2}}{2}(v_{n-1}+v_n),\qquad e_n=\frac{\sqrt{2}}{2}(v_{n-1}-v_n),\] 
for $i=1,\ldots,n-2$. 
The non-trivial Lie brackets of $\mathfrak g=\mathfrak h\rtimes_D\mathbb R$ in this basis are
\begin{align*}
[e_i,e_{n-1}]&=-\frac{\sqrt{2}}{2}\sum_{j=1}^{n-2}d_{ji}e_j-\frac{d_{n-1,i}}{2}(e_{n-1}+e_n),\\
[e_i,e_n]&=\frac{\sqrt{2}}{2}\sum_{j=1}^{n-2}d_{ji}e_j+\frac{d_{n-1,i}}{2}(e_{n-1}+e_n),\\
[e_{n-1},e_n]&=\sum_{j=1}^{n-2}d_{j,n-1}e_j+\frac{\sqrt{2}d_{n-1,n-1}}{2}(e_{n-1}+e_n),
\end{align*} 
for all $i,j=1,\dots,n-2$. The differentials of the dual basis $\{e^1,\dots,e^n\}$ are given by 
\begin{align}
\label{eq:differentials_isotropic1}
\sqrt 2\dd e^i&=\sum_{j=1}^{n-2}d_{ij}e^j\wedge e^{n-1}-\sum_{j=1}^{n-2}d_{ij}e^j\wedge e^n-\sqrt 2d_{i,n-1}e^{n-1}\wedge e^n, \\
\label{eq:differentials_isotropic2}
2\dd e^{\ell}&= \sum_{j=1}^{n-2}d_{n-1,j}e^j\wedge e^{n-1}-\sum_{j=1}^{n-2}d_{n-1,j}e^j\wedge e^n -\sqrt{2}d_{n-1,n-1}e^{n-1}\wedge e^n, 
\end{align}
for all $i=1,\ldots,n-2$, and $\ell\in \{n-1,n\}$. 
We then have the following formula for the Dirac operator in the isotropic case.

\begin{proposition}
\label{prop:Dirac_almost_Abelian_isotropic}
Let $(\mathfrak g=\mathfrak h\rtimes_D\mathbb R,g)$ be an isotropic almost Abelian pseudo-Riemannian Lie algebra and $\psi\in\Sigma$ a left-invariant spinor. Then
\begin{align*}
\slashed{D}\psi&=-\frac{\sqrt{2}}{8}\sum_{1\leq i<j\leq n-2}(\varepsilon_id_{ij}-\varepsilon_jd_{ji})(e^i\wedge e^j\wedge(e^{n-1}-e^n))\psi\\
&\quad+\frac{1}{4}\sum_{i=1}^{n-2}\varepsilon_id_{i,n-1}(e^i\wedge e^{n-1}\wedge e^n)\psi-\frac{\sqrt{2}}{4}\mathrm{tr}(D)(e^{n-1}-e^n)\psi,
\end{align*} 
where $D=(d_{ij})$ and $\varepsilon_i=g(e_i,e_i)\in\{\pm1\}$.
\end{proposition}

\begin{proof}
The general formula for the Dirac operator is given by Proposition \ref{prop:formula_Dirac_operator}. 
We compute each term separately using \eqref{eq:differentials_isotropic1}--\eqref{eq:differentials_isotropic2}. 
First we compute 
\begin{align*}
\iota_{e_i}\dd e^i&=\frac{\sqrt{2}}{2}d_{ii}e^{n-1}-\frac{\sqrt{2}}{2}d_{ii}e^n,\\
\iota_{e_{n-1}}\dd e^{n-1}&=-\frac{1}{2}\sum_{j=1}^{n-2}d_{n-1,j}e^j-\frac{\sqrt{2}d_{n-1,n-1}}{2}e^n,\\
\iota_{e_n}\dd e^n&=\frac{1}{2}\sum_{j=1}^{n-2}d_{n-1,j}e^j+\frac{\sqrt{2}d_{n-1,n-1}}{2}e^{n-1}.
\end{align*}
Summing all terms we obtain 
\[\sum_{i=1}^n\iota_{e_i}\dd e^i=\frac{\sqrt{2}}{2}\mathrm{tr}(D)(e^{n-1}-e^n).\]
Now we compute the second term. 
First note that $\varepsilon_{n-1}=1$ and $\varepsilon_n=-1$. 
Hence, we can easily show that $\varepsilon_{n-1}e^{n-1}\wedge\dd e^{n-1}+\varepsilon_ne^n\wedge\dd e^n=0$. 
We also have 
\begin{align*}
\varepsilon_ie^i\wedge\dd e^i&=\frac{\sqrt{2}}{2}\sum_{j=1}^{n-2}\varepsilon_id_{ij}e^i\wedge e^j\wedge e^{n-1}-\frac{\sqrt{2}}{2}\sum_{j=1}^{n-2}\varepsilon_id_{ij}e^i\wedge e^j\wedge e^n \\
& \qquad -\varepsilon_id_{i,n-1}e^i\wedge e^{n-1}\wedge e^n.
\end{align*}
Summing over $i$ we obtain 
\begin{align*}
\sum_{i=1}^{n-2}\varepsilon_ie^i\wedge\dd e^i&=\frac{\sqrt{2}}{2}\sum_{1\leq i<j\leq n-2}(\varepsilon_id_{ij}-\varepsilon_jd_{ji})e^i\wedge e^j\wedge e^{n-1}\\
&\quad-\frac{\sqrt{2}}{2}\sum_{1\leq i<j\leq n-2}(\varepsilon_id_{ij}-\varepsilon_jd_{ji})e^i\wedge e^j\wedge e^n\\
&\quad-\sum_{i=1}^{n-2}\varepsilon_id_{i,n-1}e^i\wedge e^{n-1}\wedge e^n.
\end{align*}
Putting everything together we obtain the result.
\end{proof}

\subsection{Harmonic spinors on two-dimensional Lie groups}
\label{subsec:harmonic-spinors-two-dim}
Let $\mathfrak g$ be a two-dimensional almost Abelian Lie algebra, that is, $\mathfrak g=\mathbb R\rtimes_D\mathbb R$ for $D=d\in\mathbb R\setminus\{0\}$. 
Then $\mathfrak g=\mathfrak{aff}(\mathbb R)$ is non-unimodular since $\mathrm{tr}(D)=d$. 
Hence, by Corollary \ref{cor:almost_Abelian_non_unimodular}, we know that $(\mathfrak g,g)$ does not admit left-invariant harmonic spinors for any Riemannian metric $g$.

Let $g$ be a Lorentzian metric on $\mathfrak g$ and let $\{e_1,e_2\}$ be an orthonormal basis of $(\mathfrak g,g)$ with $\varepsilon_1=-\varepsilon_2=1$. 
Consider the Clifford algebra $\mathrm{Cl}(1,1)$ generated by $\{e^1,e^2\}$ and satisfying the relations 
\[e^ie^j+e^je^i=-2g(e^i,e^j), \qquad i=1,2,\] 
Recall that the standard representation $\Sigma$ of $\mathrm{Cl}(1,1)$ is complex and $\Sigma=\mathbb C^2$.
In the non-isotropic case, we use Proposition \ref{prop:Dirac_almost_Abelian} to compute the Dirac operator acting on a left-invariant spinor $\psi\in\Sigma$: 
\[\slashed{D}\psi=-\frac{1}{2}\mathrm{tr}(D)e^2\psi.\]
Since the endomorphism associated to $e^2$ squares to the identity map on $\Sigma$, the Dirac operator $\slashed{D}\in\mathrm{End}_{\mathbb C}(\Sigma)$ is invertible, so it has no kernel and left-invariant harmonic spinors do not exist. 
In the isotropic case, we use Proposition \ref{prop:Dirac_almost_Abelian_isotropic}: 
\[\slashed{D}\psi=-\frac{\sqrt{2}}{4}\mathrm{tr}(D)(e^1-e^2)\psi.\]
Since $(e^1-e^2)^2=0$ and $\Sigma=\mathbb C^2$, then the endomorphism $\slashed{D}\in\mathrm{End}_{\mathbb C}(\Sigma)$ has a one-dimensional kernel, and left-invariant harmonic spinors exist. 
Moreover, a direct computation shows that in this case the Lorentzian Lie algebra $(\mathfrak g,g=v^1\odot v^2)$ is flat.

We now show that the above metric $g$ is the unique metric on $\mathfrak{aff}(\mathbb R)$ that admits left-invariant harmonic spinors. 
The Lie algebra $\mathfrak g$ is spanned by $\{v_1,v_2\}$ satisfying $[v_1,v_2]=dv_1$ and the flat metric is $g=v^1\odot v^2$. 
By the change of basis $v_1\mapsto v_1$ and $v_2\mapsto v_2/d$, the Lie bracket becomes $[v_1,v_2]=v_1$ and the metric becomes $g=v_1\odot v_2/d$, which is equivalent to $v_1\odot v_2$. 
In fact, we have the following classification.

\begin{proposition}
Any Lorentzian metric on $\mathfrak g=\mathfrak{aff}(\mathbb R)$ is equivalent to a metric whose associated matrix with respect to the basis $\{v_1,v_2\}$ is of the following form: 
\[g_+(t)=\begin{pmatrix}1&0\\0&-t\end{pmatrix},\qquad g_-(t)=\begin{pmatrix}-1&0\\0&t\end{pmatrix},\qquad g_0=\begin{pmatrix}0&1\\1&0\end{pmatrix},\] 
where $t>0$. The metrics $g_{\pm}(t)$ have $\mathrm{Scal}(g_{\pm}(t))=\pm 2/t$ and the metric $g_0$ is flat.
\end{proposition}
\begin{proof}
It is easily shown that
\[\mathrm{Aut}(\mathfrak g)\coloneqq\{A\in\mathrm{GL}(\mathfrak g)\mid A[x,y]=[Ax,Ay], \text{ for }x,y\in\mathfrak g\}=\left\{\begin{pmatrix}a&b\\0&1\end{pmatrix}\mid a\neq0\right\},\]
where elements in the latter set are expressed with respect to the ordered basis $\{v_1,v_2\}$ such that $[v_1,v_2]=v_1$.

Let $g=\left(\begin{smallmatrix}x&y\\y&z\end{smallmatrix}\right)$ be a Lorentzian metric on $\mathfrak g$, thus $xz<y^2$, and let $A\in\mathrm{Aut}(\mathfrak g)$ be an automorphism of $\mathfrak g$. 
Then 
\[A^TgA=\begin{pmatrix}a&0\\b&1\end{pmatrix}\begin{pmatrix}x&y\\y&z\end{pmatrix}\begin{pmatrix}a&b\\0&1\end{pmatrix}=\begin{pmatrix}a^2x&abx+ay\\abx+ay&b^2x+2by+z\end{pmatrix}.\]
We now have two cases.
\begin{enumerate}
\item If $x\neq 0$, choose $a=1/\sqrt{|x|}$, $b=-y/x$, and set $t= (y^2-xz)/|x|>0$.
Then $g$ is equivalent to $g_{\mathrm{sgn}(x)}(t)$. 
It is easy to check that two metrics with $t_1\neq t_2$ are not equivalent. 
\item If $x=0$, choose $a=1/y$ and $b=-z/2y$.
Then $g$ is equivalent to $g_0$. 
\end{enumerate}
A computation in each case gives the curvature conditions in the statement.
\end{proof}
We have proved the following result.
\begin{theorem}
\label{thm:main-thm-two-dim}
The metric Lie algebra $(\mathfrak g=\mathfrak{aff}(\mathbb R),g)$ admits left-invariant harmonic spinors if and only if $g$ is equivalent to the Lorentzian metric $g_0$. Moreover, the space of left-invariant harmonic spinors is one-dimensional.
\end{theorem}

\section{Harmonic spinors on three-dimensional Riemannian Lie groups}
\label{sec:harmonic-spinors-riem-three-dim}

We start looking at the three-dimensional case. 
We find equivalent conditions for a Lie group to admit left-invariant harmonic spinors in terms of constraints on the structure equations for the corresponding Lie algebras. 
In particular, in the Riemannian case we find that Lie algebras admitting left-invariant harmonic spinors are necessarily unimodular. 
As observed in Section \ref{subsec:four-dim-non-uni-ex}, this fact does not generalise to dimension four. 
We then proceed with the classification of those Lie algebras and metrics carrying harmonic spinors. 
We discuss the Lorentzian case in the next section.

\subsection{General results}
\label{subsec:general-results}
Let $(\mathfrak g,g)$ be a three-dimensional Riemannian Lie algebra. Let $\{e_1,e_2,e_3\}$ be an orthonormal basis of $(\mathfrak g,g)$, and $\{e^1,e^2,e^3\}$ be its dual. 
Write
\begin{equation}
\label{eq:str-eq-setup}
\dd e^1= \sum_{i<j}a_{ij}e^i\wedge e^j, \qquad \dd e^2= \sum_{i<j}b_{ij}e^i\wedge e^j, \qquad \dd e^3=\sum_{i<j}c_{ij}e^i\wedge e^j, 
\end{equation} 
where $a_{ij},b_{ij},c_{ij} \in \mathbb R$, $1\leq i<j\leq3$. The condition $\dd^2=0$ is equivalent to the identities
\begin{equation}
\label{eq:d^2=0}
\begin{cases}
a_{23}b_{12}-a_{12}b_{23}+a_{23}c_{13}-a_{13}c_{23}=0,\\
a_{13}b_{12}-a_{12}b_{13}+b_{13} c_{23}-b_{23}c_{13}=0,\\
a_{13}c_{12}-a_{12}c_{13}+b_{23}c_{12}-b_{12}c_{23}=0.
\end{cases}
\end{equation}
The standard representation of the Clifford algebra $\mathrm{Cl}(3)$ allows us to identify $e^1,e^2,e^3$ with the following endomorphisms of $\Sigma = \mathbb C^2$: 
\[e^1=\begin{pmatrix}0&i\\i&0\end{pmatrix}, \qquad e^2=\begin{pmatrix}0&1\\-1&0\end{pmatrix}, \qquad e^3=\begin{pmatrix}i&0\\0&-i\end{pmatrix}.\]
Using Proposition \ref{prop:formula_Dirac_operator} in the Riemannian case, we compute the Dirac operator acting on a left-invariant spinor $\psi\in\Sigma$:  
\begin{align*}
-4\slashed{D}\psi&=(a_{23}-b_{13}+c_{12})(e^1\wedge e^2\wedge e^3)\psi\\
&\quad-2(b_{12}+c_{13})e^1\psi+2(a_{12}-c_{23})e^2\psi+2(a_{13}+b_{23})e^3\psi.
\end{align*}
The above matrix representations of $e^1,e^2,e^3$ give 
\[-4\slashed{D}\psi=\begin{pmatrix}a_{23}-b_{13}+c_{12}+2i(a_{13}+b_{23})&2(a_{12}-c_{23})-2i(b_{12}+c_{13})\\-2(a_{12}-c_{23})-2i(b_{12}+c_{13})&a_{23}-b_{13}+c_{12}-2i(a_{13}+b_{23})\end{pmatrix}\psi.\]

Note that the determinant of the $2\times 2$ matrix on the right-hand side vanishes if and only if the whole matrix vanishes. 
Hence there are left-invariant harmonic spinors when 
\[b_{23}=-a_{13}, \qquad c_{12}=-a_{23}+b_{13}, \qquad c_{13}=-b_{12}, \qquad c_{23}=a_{12},\] 
and every left-invariant spinor $\psi\in\Sigma$ is harmonic. 
Note also that these relations are solutions to the system of equations \eqref{eq:d^2=0}. 
We have proved the following result.
\begin{proposition}\label{prop:3d_Riemmanian_harmonic}
A three-dimensional Riemannian Lie algebra $(\mathfrak g,g)$ admits left-invariant harmonic spinors if and only if \begin{align*}
\dd e^1&=a_{12}e^1\wedge e^2+a_{13}e^1\wedge e^3+a_{23}e^2\wedge e^3,\\
\dd e^2&=b_{12}e^1\wedge e^2+b_{13}e^1\wedge e^3-a_{13}e^2\wedge e^3,\\
\dd e^3&=(b_{13}-a_{23})e^1\wedge e^2-b_{12}e^1\wedge e^3+a_{12}e^2\wedge e^3,
\end{align*} where $a_{12},a_{13},a_{23},b_{12},b_{13}\in\mathbb R$, and $\{e^1,e^2,e^3\}$ is an orthonormal basis. In this case, the space of left-invariant harmonic spinors is two-dimensional.
\end{proposition}
An equivalent way to express the equations in Proposition \ref{prop:3d_Riemmanian_harmonic} is via the Lie brackets
\begin{equation}
\label{eq:Lie_brackets_harmonic}
\begin{aligned}
[e_1,e_2]&=-a_{12}e_1-b_{12}e_2+(a_{23}-b_{13})e_3,\\
[e_1,e_3]&=-a_{13}e_1-b_{13}e_2+b_{12}e_3,\\
[e_2,e_3]&=-a_{23}e_1+a_{13}e_2-a_{12}e_3.
\end{aligned}
\end{equation}

Now, consider the following matrix constructed using the structure constants of the Lie algebra $\mathfrak g$ in \eqref{eq:Lie_brackets_harmonic}: 
\[C\coloneqq\begin{pmatrix}-a_{12}&-a_{13}&-a_{23}\\-b_{12}&-b_{13}&a_{13}\\a_{23}-b_{13}&b_{12}&-a_{12}\end{pmatrix}.\]
\begin{lemma}
\label{lemma:cmatrix}
Let $(\mathfrak g,g)$ be a three-dimensional Riemannian non-Abelian Lie algebra admitting left-invariant harmonic spinors. 
\begin{enumerate}
\item If $\det(C)\neq0$, then $\mathfrak g\cong\mathfrak{sl}(2,\mathbb R)$.
\item If $\det(C)=0$, then $\mathfrak g\cong\mathfrak{e}(1,1)$.
\end{enumerate}
\end{lemma}
\begin{proof}
It is known that a Lie algebra is unimodular if and only if $\dd(\Lambda^{n-1}\mathfrak g^*)=0$, where $n$ is the dimension of $\mathfrak g$. 
Using Proposition \ref{prop:3d_Riemmanian_harmonic} we check that $\dd(\Lambda^2\mathfrak g^*)=0$. 
Hence, the Lie algebra $\mathfrak g$ is unimodular.
    
Assume now that $\mathfrak g$ is nilpotent.  
Since $\mathfrak g$ is three-dimensional, the Lie algebra $\mathfrak g$ has to be two-step nilpotent, namely $[\mathfrak g,[\mathfrak g,\mathfrak g]]=0$. 
From $[e_2,[e_2,e_3]]=0$ and $[e_3,[e_2,e_3]]=0$ we get in particular that $a_{12}^2+a_{23}^2-a_{23}b_{13}=0$ and $a_{13}^2+a_{23}b_{13}=0$. 
Summing both terms we obtain $a_{12}^2+a_{13}^2+a_{23}^2=0$, thus $a_{12}=a_{13}=a_{23}=0$. 
Now we have 
\[0=[e_1,[e_1,e_2]]=(b_{12}^2+b_{13}^2)e_2,\] 
so $b_{12}=b_{13}=0$ and the Lie algebra $\mathfrak g$ is Abelian, but we are neglecting this case. 
Therefore, a non-Abelian three-dimensional Riemannian Lie algebra $(\mathfrak g,g)$ admitting left-invariant harmonic spinors must be unimodular and non-nilpotent. 
Then it can be $\mathfrak{su}(2)$, $\mathfrak{sl}(2,\mathbb R)$, $\mathfrak{e}(2)$, or $\mathfrak{e}(1,1)$.
    
The Killing form $B_\mathfrak g$ of the Lie algebra $\mathfrak g$ is 
\begin{equation*}
\label{eq:Killing_Riemannian}
B_\mathfrak g=
2\left(\begin{smallmatrix}
b_{12}^{2} - a_{23} b_{13} + b_{13}^{2} & a_{13} a_{23} - a_{12} b_{12} - a_{13} b_{13} & a_{13} b_{12} - a_{12} b_{13} \\
a_{13} a_{23} - a_{12} b_{12} - a_{13} b_{13} & a_{12}^{2} + a_{23}^{2} - a_{23} b_{13} & a_{12} a_{13} + a_{23} b_{12} \\
a_{13} b_{12} - a_{12} b_{13} & a_{12} a_{13} + a_{23} b_{12} & a_{13}^{2} + a_{23} b_{13}
\end{smallmatrix}\right).
\end{equation*}
The trace of $B_\mathfrak g$ is 
\[\mathrm{tr}_g(B_\mathfrak g)=2(a_{12}^{2}+a_{13}^{2}+a_{23}^{2}-a_{23}b_{13}+b_{12}^{2}+b_{13}^{2}).\]
Since $-a_{23}b_{13}\geq-\frac{1}{2}(a_{23}^2+b_{13}^2)$ we get 
\[\frac12\mathrm{tr}_g(B_\mathfrak g)\geq a_{12}^{2}+a_{13}^{2}+b_{12}^{2}+\frac{1}{2}a_{23}^{2}+\frac{1}{2}b_{13}^{2}>0,\]
as $\mathfrak g$ is non-Abelian.
Hence, $\mathrm{tr}_g(B_\mathfrak g)>0$, which rules out $\mathfrak{su}(2)$ and $\mathfrak{e}(2)$.

Finally, note that $\dim_{\mathbb R}([\mathfrak g,\mathfrak g])=\mathrm{rank}(C)$. 
If $\det(C)\neq0$, then $\dim_{\mathbb R}([\mathfrak g,\mathfrak g])=3$, which implies that $\mathfrak g$ is simple, so it is $\mathfrak{sl}(2,\mathbb R)$. 
If $\det(C)=0$, then $\dim_{\mathbb R}([\mathfrak g,\mathfrak g])<3$, thus $\mathfrak g$ is $\mathfrak{e}(1,1)$.
\end{proof}
We have identified those three-dimensional Lie algebras admitting left-invariant harmonic spinors. 
We now detect the Riemannian metrics on these Lie algebras carrying left-invariant harmonic spinors. 
First we obtain the following restriction on the Ricci curvature of such metrics. 
\begin{lemma}
\label{lemma:Ric=-B_Riemannian}
Let $(\mathfrak g,g)$ be a three-dimensional Riemannian non-Abelian Lie algebra admitting left-invariant harmonic spinors. 
Then the Ricci curvature of $g$ satisfies $\mathrm{Ric}(g)=-B_\mathfrak g$.
In particular, if $\sigma(\mathfrak g)=(p,q,r)$, then $\sigma(\mathrm{Ric}(g))=(q,p,r)$.
\end{lemma}
\begin{proof}
Let $\{e_1,\ldots,e_n\}$ be an orthonormal basis of a pseudo-Riemannian Lie algebra $(\mathfrak{g},g)$ of dimension $n$. Denote by $z$ the vector on $\mathfrak{g}$ defined by $g(z,{}\cdot{})=\mathrm{tr}(\mathrm{ad}({}\cdot{}))$. Then the following formula for the Ricci curvature $\mathrm{Ric}(g)$ of $g$ holds (see e.g.\ \cite[Lemma 1.1 and Remark 1.3]{conti-rossi}): \begin{align*}
    \mathrm{Ric}(g)(x,x)&=-\frac{1}{2}\sum_{i=1}^n\varepsilon_ig([x,e_i],[x,e_i])+\frac{1}{2}\sum_{i<j}\varepsilon_i\varepsilon_jg([e_i,e_j],x)^2\\
    &\quad-\frac{1}{2}B_\mathfrak{g}(x,x)-g([z,x],x),
\end{align*} for all $x\in\mathfrak{g}$.
In our case, a straightforward computation using \eqref{eq:Lie_brackets_harmonic} shows that $z=0$ and $\mathrm{Ric}(g)=-B_{\mathfrak{g}}$.
\end{proof}

We recall the following classification results.
\begin{theorem}[Ha--Lee \cite{ha-lee1}]
Let $\mathfrak{sl}(2,\mathbb R)$ be spanned by $\{x,y,z\}$ satisfying 
\[[x,y]=2z, \qquad [z,x]=2y, \qquad [z,y]=2x.\]
Any Riemannian metric on $\mathfrak{sl}(2,\mathbb R)$ is equivalent to a metric whose associated matrix with respect to the basis $\{x,y,z\}$ is of the form 
\[g(\lambda,\mu,\nu)=\begin{pmatrix}\lambda&0&0\\0&\mu&0\\0&0&\nu\end{pmatrix}, \qquad \lambda>0, \quad \mu\geq\nu>0.\] 
Further, the signature of the corresponding Ricci tensors is 
\[\sigma(\mathrm{Ric}(g(\lambda,\mu,\nu)))=
\begin{cases}
(0,1,2)&\mbox{ if }\lambda=\mu-\nu,\\
(1,2,0)&\mbox{ if }\lambda\neq\mu-\nu.
\end{cases}
\]
\end{theorem}
\begin{remark}
The condition $\lambda=\mu+\nu$ in \cite[Thm.\ 4.8]{ha-lee1} seems to be a typo. 
Inspecting the proof of that theorem, we observe that the correct condition is $\lambda=\mu-\nu$.
\end{remark}
\begin{theorem}[Ha--Lee \cite{ha-lee1}]
Let $\mathfrak{e}(1,1)$ be spanned by $\{x,y,z\}$ satisfying 
\[[x,y]=0, \qquad [z,x]=x, \qquad [z,y]=-y.\] 
Any Riemannian metric on $\mathfrak{e}(1,1)$ is equivalent to a metric whose associated matrix with respect to the basis $\{x,y,z\}$ is of the form 
\[g(\nu)=\begin{pmatrix}1&0&0\\0&1&0\\0&0&\nu\end{pmatrix},\qquad g(\mu,\nu)=\begin{pmatrix}1&1&0\\1&\mu&0\\0&0&\nu\end{pmatrix},\]
with $\nu>0$, $\mu>1$. 
Further, the signature of the corresponding Ricci tensors is 
\[\sigma(\mathrm{Ric}(g(\nu)))=(0,1,2),\qquad \sigma(\mathrm{Ric}(g(\mu,\nu)))=(1,2,0).\]
\end{theorem}
The following is the main result of this section, cf.\ Table \ref{table:Riemannian}.
\begin{theorem}
\label{thm:3d_Riemannian_classification}
Let $(\mathfrak g,g)$ be a three-dimensional Riemannian non-Abelian Lie algebra admitting left-invariant harmonic spinors. Then
\begin{enumerate}
\item $\mathfrak g\cong\mathfrak{sl}(2,\mathbb R)$ and $g$ is equivalent to $g(\mu+\nu,\mu,\nu)$, $\mu\geq\nu>0$, or
\item $\mathfrak g\cong\mathfrak{e}(1,1)$ and $g$ is equivalent to $g(\nu)$, $\nu>0$.
\end{enumerate}
\end{theorem}
\begin{proof}
The two cases follow by Lemma \ref{lemma:cmatrix}, so we only need to detect the metrics carrying left-invariant harmonic spinors.

Suppose $\mathfrak g \cong \mathfrak{sl}(2,\mathbb R)$. 
By Lemma \ref{lemma:Ric=-B_Riemannian}, the Ricci curvature $\mathrm{Ric}(g)$ has signature $(1,2,0)$ (cf.\ Table \ref{table:three-dim-lie-algebras-uni}). 
Take the Lie algebra $\mathfrak{sl}(2,\mathbb R)$ with basis $\{x,y,z\}$ and metric $g(\lambda,\mu,\nu)$, and write $e_1=x/\sqrt{\lambda}$, $e_2=y/\sqrt{\mu}$, and $e_3=z/\sqrt{\nu}$. 
Then $\{e_1,e_2,e_3\}$ is an orthonormal basis and
\[[e_1,e_2]=\frac{2\nu}{\sqrt{\lambda\mu\nu}}e_3,\qquad [e_1,e_3]=-\frac{2\mu}{\sqrt{\lambda\mu\nu}}e_2,\qquad [e_2,e_3]=-\frac{2\lambda}{\sqrt{\lambda\mu\nu}}e_1.\]
Comparing with \eqref{eq:Lie_brackets_harmonic} we get $a_{12}=a_{13}=b_{12}=0$, $a_{23}=2\lambda/\sqrt{\lambda\mu\nu}$, $b_{13}=2\mu/\sqrt{\lambda\mu\nu}$, and $\lambda=\mu+\nu$.

Suppose now $\mathfrak g\cong\mathfrak{e}(1,1)$. 
By Lemma \ref{lemma:Ric=-B_Riemannian}, the Ricci curvature $\mathrm{Ric}(g)$ has signature $(0,1,2)$ (cf.\ Table \ref{table:three-dim-lie-algebras-uni}). 
Take the Lie algebra $\mathfrak{e}(1,1)$ with basis $\{x,y,z\}$ and metric $g(\nu)$, and write $e_1=(x+y)/\sqrt 2$, $e_2=(y-x)/\sqrt 2$, and $e_3=z/\sqrt \nu$. 
Then $\{e_1,e_2,e_3\}$ is an orthonormal basis and
\[[e_1,e_2]=0,\qquad [e_1,e_3]=\frac{1}{\sqrt{\nu}}e_2,\qquad [e_2,e_3]=\frac{1}{\sqrt{\nu}}e_1.\]
Comparing with \eqref{eq:Lie_brackets_harmonic} we get $a_{12}=a_{13}=b_{12}=0$ and $a_{23}=b_{13}=-1/\sqrt{\nu}$.
\end{proof}

For completeness, we now check that $\mathfrak{e}(1,1)$ is the unique three-dimensional almost Abelian Lie algebra admitting a Riemannian metric carrying left-invariant harmonic spinors.
We use Proposition \ref{prop:Dirac_almost_Abelian} to do this.
\begin{proposition}
Let $(\mathfrak g=\mathfrak h\rtimes_D\mathbb R,g)$ be a three-dimensional almost Abelian Riemannian Lie algebra. 
Then it admits left-invariant harmonic spinors if and only if $\mathfrak g\cong\mathfrak{e}(1,1)$ and $g$ is equivalent to $g(\nu)$ for $\nu>0$.
\end{proposition}
\begin{proof}
Let $\{e_1,e_2,e_3\}$ be an orthonormal basis of $(\mathfrak g,g)$, $\mathfrak h=\mathrm{Span}_{\mathbb R}\{e_1,e_2\}$, and 
\[D=\mathrm{ad}(e_3)|_{\mathfrak h}=\begin{pmatrix}d_{11}&d_{12}\\d_{21}&d_{22}\end{pmatrix}\] 
for $d_{ij}\in\mathbb R$. 
First, if $\mathrm{tr}(D)\neq0$, then $(\mathfrak g,g)$ does not admit left-invariant harmonic spinors by Corollary \ref{cor:almost_Abelian_non_unimodular}. 
Therefore $d_{22}=-d_{11}$. 
Let $\psi\in\Sigma$ be a left-invariant spinor. 
By Proposition \ref{prop:Dirac_almost_Abelian} we have 
\[\slashed{D}\psi=-\frac{1}{4}(d_{12}-d_{21})(e^1\wedge e^2\wedge e^3)\psi=-\frac{1}{4}(d_{12}-d_{21})\psi=0\] if and only if $d_{21}=d_{12}$. 
That is, the endomorphism $D$ is equivalent to
\[\sqrt{d_{11}^2+d_{12}^2}\begin{pmatrix}-1&0\\0&1\end{pmatrix},\] 
which implies that $\mathfrak g$ is completely solvable, therefore $\mathfrak g\cong\mathfrak{e}(1,1)$. 
A computation shows that in this situation we have $\sigma(\mathrm{Ric}(g))=(0,1,2)$, as expected. 
Set $t\coloneqq\sqrt{d_{11}^2+d_{12}^2}>0$ and choose $\theta\in\mathbb R$ such that $\cos(2\theta)=d_{11}/t$ and $\sin(2\theta)=d_{12}/t$. 
Take \[x=-\sin(\theta)e_1+\cos(\theta)e_2,\qquad y=\cos(\theta)e_1+\sin(\theta)e_2,\qquad z=-\frac{1}{t}e_3.\] 
The basis $\{x,y,z\}$ satisfies $[x,y]=0$, $[x,z]=-x$, and $[y,z]=y$. 
The matrix of the metric $g$ in this new basis is 
\[\begin{pmatrix}1&0&0\\0&1&0\\ 0&0&1/t^2\end{pmatrix},\]
which gives the metric $g(\nu)$ with $\nu=1/t^2$.
\end{proof}

\subsection{A four-dimensional non-unimodular example}
\label{subsec:four-dim-non-uni-ex}
In Lemma~\ref{lemma:cmatrix} we have seen that if a three-dimensional Riemannian Lie algebra $(\mathfrak g,g)$ admits left-invariant harmonic spinors, then $\mathfrak g$ is unimodular. 
We conclude this section by exhibiting a four-dimensional example where this does not hold.

Let us consider the Lie algebra $\mathfrak g$ defined by 
\[\dd e^1=e^1\wedge e^4+4e^2\wedge e^3,\qquad \dd e^2=e^2\wedge e^4,\qquad \dd e^3=0,\qquad \dd e^4=0,\] 
where $\{e^1,e^2,e^3,e^4\}$ is the dual basis of the orthonormal basis $\{e_1,e_2,e_3,e_4\}$ of $(\mathfrak g,g)$. 
Since 
\[\dd(e^1\wedge e^2\wedge e^3)=2e^1\wedge e^2\wedge e^3\wedge e^4,\] 
this Lie algebra is non-unimodular. 
The basis vectors $e^1,e^2,e^3,e^4$ are identified with the following endomorphisms of $\Sigma=\mathbb C^4$, now a representation of $\mathrm{Cl}(4)$: 
\[
e^1 =
\left(\begin{smallmatrix}
    0 & 0 & 0 & -1 \\
    0 & 0 & -1 & 0 \\
    0 & 1 & 0 & 0 \\
    1 & 0 & 0 & 0
\end{smallmatrix}\right),
\quad 
e^2=\left(\begin{smallmatrix}
    0 & 0 & 0 & i \\
    0 & 0 & -i & 0 \\
    0 & -i & 0 & 0 \\
    i & 0 & 0 & 0
\end{smallmatrix}\right), \quad
e^3=\left(\begin{smallmatrix}
    0 & 0 & -1 & 0 \\
    0 & 0 & 0 & 1 \\
    1 & 0 & 0 & 0 \\
    0 & -1 & 0 & 0
\end{smallmatrix}\right),
\quad  
e^4=\left(\begin{smallmatrix}
    0 & 0 & i & 0 \\
    0 & 0 & 0 & i \\
    i & 0 & 0 & 0 \\
    0 & i & 0 & 0
\end{smallmatrix}\right).
\]
Using Proposition \ref{prop:formula_Dirac_operator} we compute 
\begin{align*}
-4\slashed{D}\psi=4(e^1\wedge e^2\wedge e^3+e^4)\psi=
8i\left(\begin{smallmatrix}
        0&0&1&0\\
        0&0&0&1\\
        0&0&0&0\\
        0&0&0&0
\end{smallmatrix}\right)\psi.
\end{align*}
Since $\det(\slashed{D})=0$, then $(\mathfrak g,g)$ admits left-invariant harmonic spinors.

Let us identify this four-dimensional non-unimodular Lie algebra. 
Consider the map $\chi\colon\mathfrak g\to\mathbb R$ defined by 
\[\chi(x)\coloneqq\mathrm{tr}(\mathrm{ad}(x))\] 
for all $x\in\mathfrak g$. 
The kernel of $\chi$ is an ideal of $\mathfrak g$ which contains the derived subalgebra $[\mathfrak g,\mathfrak g]$. 
It is called the \emph{unimodular kernel} of $\mathfrak g$ and it is a unimodular Lie algebra itself. 
In our case, we have that 
\[\ker(\chi)=\mathrm{Span}_{\mathbb R}\{e_1,e_2,e_3\}\cong\mathfrak{heis}_3.\]
Setting $f_0\coloneqq e_4$, $f_1\coloneqq e_2$, $f_2\coloneqq e_3$, and $f_3\coloneqq-4e_1$, we conclude that the Lie algebra $\mathfrak g$ is isomorphic to the Lie algebra denoted by $\mathfrak{d}_{4,1}$ in \cite{andrada-barberis-dotti-ovando}.

\section{Harmonic spinors on three-dimensional Lorentzian Lie groups}
\label{sec:harmonic-spinors-three-dimensional-lorentzian-groups}

We now carry out the same program as in Section \ref{sec:harmonic-spinors-riem-three-dim} in the Lorentzian case.
We will distinguish between the unimodular and the non-unimodular cases as the existence of left-invariant harmonic spinors is now equivalent to various non-trivial algebraic conditions.
Unlike in Section \ref{sec:harmonic-spinors-riem-three-dim}, there will be many more metrics to be discussed due to the abundance of the various classification results involved.
To simplify the exposition, we will present only the essential details of the arguments.

Let $(\mathfrak g,g)$ be a three-dimensional Lorentzian Lie algebra, $\{e_1,e_2,e_3\}$ an orthonormal basis of $(\mathfrak g,g)$ satisfying $\varepsilon_1=\varepsilon_2=-\varepsilon_3=1$, and $\{e^1,e^2,e^3\}$ its dual basis. 
We assume to be in the same set-up as in \eqref{eq:str-eq-setup}--\eqref{eq:d^2=0}.
Consider the Clifford algebra $\mathrm{Cl}(2,1)$, generated by $\{e^1,e^2,e^3\}$. 
The basis vectors $e^1,e^2,e^3$ are identified with the following endomorphisms of $\Sigma=\mathbb C^2$: 
\begin{equation}
\label{eq:rep_Cl(2,1)}
e^1=\begin{pmatrix}
    0&i\\
    i&0
\end{pmatrix},\qquad e^2=\begin{pmatrix}
    0&1\\
    -1&0
\end{pmatrix},\qquad e^3=\begin{pmatrix}
    1&0\\
    0&-1
\end{pmatrix}.
\end{equation}
Note that $\varepsilon_i=g(e^i,e^i) = 1$ if $i=1,2$, and $\varepsilon_3=g(e^3,e^3)=-1$.
We compute the Dirac operator acting on a left-invariant spinor $\psi\in\Sigma$ by using Proposition \ref{prop:formula_Dirac_operator}: 
\begin{align*}
-4\slashed{D}\psi&=\sum_{i=1}^3(\varepsilon_ie^i\wedge\dd e^i+2\iota_{e_i}\dd e^i)\psi\\
&=(a_{23}-b_{13}-c_{12})(e^1\wedge e^2\wedge e^3)\psi\\
&\quad-2(b_{12}+c_{13})e^1\psi+2(a_{12}-c_{23})e^2\psi+2(a_{13}+b_{23})e^3\psi.
\end{align*}
Using the matrix representations in \eqref{eq:rep_Cl(2,1)} we get 
\[4\slashed{D}\psi=\begin{pmatrix}
    -2(a_{13}+b_{23})+i(a_{23}-b_{13}-c_{12})&-2(a_{12}-c_{23})+2i(b_{12}+c_{13})\\
    2(a_{12}-c_{23})+2i(b_{12}+c_{13})&2(a_{13}+b_{23})+i(a_{23}-b_{13}-c_{12})
\end{pmatrix}\psi.\]
Left-invariant harmonic spinors exist when the $2\times 2$ matrix on the right-hand side has zero determinant, i.e.\
\begin{equation}
\label{eq:Dirac_Lorentzian}
4(a_{12}-c_{23})^2-4(a_{13}+b_{23})^2-(a_{23}-b_{13}-c_{12})^2+4(b_{12}+c_{13})^2=0.
\end{equation}
Let us consider the unimodular and the non-unimodular cases separately.
\subsection{Unimodular case}
\label{subsec:unimodular-case}
Recall that $\mathfrak g$ is unimodular if and only if $\dd(\Lambda^2\mathfrak g^*)=0$. We have 
\begin{align*}
\dd(e^1\wedge e^2)&=-(a_{13}+b_{23})e^1\wedge e^2\wedge e^3,\\
\dd(e^1\wedge e^3)&=(a_{12}-c_{23})e^1\wedge e^2\wedge e^3,\\
\dd(e^2\wedge e^3)&=(b_{12}+c_{13})e^1\wedge e^2\wedge e^3.
\end{align*}
Therefore, $\mathfrak g$ is unimodular if and only if $b_{23}=-a_{13}$, $c_{13}=-b_{12}$, and $c_{23}=a_{12}$. 
Note that these relations are solutions to the system \eqref{eq:d^2=0}. 
Imposing these linear conditions on \eqref{eq:Dirac_Lorentzian}, we conclude that $\mathfrak g$ admits left-invariant harmonic spinors only if in addition $c_{12}=a_{23}-b_{13}$. 
In fact, all these relations imply that $\slashed{D}$ is identically zero, thus every left-invariant spinor $\psi\in\Sigma$ is harmonic. 
We have proved the following result.
\begin{proposition}
\label{prop:3d_Lorentzian_unimodular_harmonic}
A unimodular three-dimensional Lorentzian Lie algebra $(\mathfrak g,g)$ admits left-invariant harmonic spinors if and only if 
\begin{align*}
\dd e^1&=a_{12}e^1\wedge e^2+a_{13}e^1\wedge e^3+a_{23}e^2\wedge e^3,\\
\dd e^2&=b_{12}e^1\wedge e^2+b_{13}e^1\wedge e^3-a_{13}e^2\wedge e^3,\\
\dd e^3&=(a_{23}-b_{13})e^1\wedge e^2-b_{12}e^1\wedge e^3+a_{12}e^2\wedge e^3,
\end{align*} 
where $a_{12},a_{13},a_{23},b_{12},b_{13}\in\mathbb R$ and $\{e^1,e^2,e^3\}$ is an orthonormal basis with $e^3$ timelike. 
In this case, the space of left-invariant harmonic spinors is two-dimensional.
\end{proposition}
The Lie brackets of the Lie algebra in Proposition \ref{prop:3d_Lorentzian_unimodular_harmonic} are 
\begin{equation}
\label{eq:Lie_brackets_Lor_unimodular}
\begin{aligned}
[e_1,e_2]&=-a_{12}e_1-b_{12}e_2-(a_{23}-b_{13})e_3,\\
[e_1,e_3]&=-a_{13}e_1-b_{13}e_2+b_{12}e_3,\\
[e_2,e_3]&=-a_{23}e_1+a_{13}e_2-a_{12}e_3.
\end{aligned}
\end{equation}
In order to detect the unimodular three-dimensional Lorentzian Lie algebras $(\mathfrak g,g)$ admitting left-invariant harmonic spinors, we make use of the following relation between the Killing form $B_{\mathfrak g}$ of $\mathfrak g$ and the Ricci curvature $\mathrm{Ric}(g)$ of $g$, which is obtained as in Lemma \ref{lemma:Ric=-B_Riemannian}.
\begin{lemma}
\label{lemma:Ric=-B_Lorentzian_uni}
Let $(\mathfrak g,g)$ be a unimodular three-dimensional Lorentzian Lie algebra admitting left-invariant harmonic spinors. Then the Ricci curvature of $g$ satisfies $\mathrm{Ric}(g)=-B_{\mathfrak g}$.
In particular, if $\sigma(\mathfrak g)=(p,q,r)$, then $\sigma(\mathrm{Ric}(g))=(q,p,r)$.
\end{lemma}

We now proceed to study each Lie algebra using the notation and classification results obtained in \cite{boucetta-chakkar}, which we collect in the Appendix below for the reader's convenience.
\begin{proposition}
\label{prop:su(2)-harm-sp}
Suppose that $g$ is a Lorentzian metric on $\mathfrak g=\mathfrak{su}(2)$ carrying left-invariant harmonic spinors. 
Then $g$ is equivalent to the metric $\textup{(su)}$ with $\mu_3=\mu_1+\mu_2$ for $\mu_1\geq\mu_2>0$.
\end{proposition}
\begin{proof}
By Lemma \ref{lemma:Ric=-B_Lorentzian_uni} and Table \ref{table:three-dim-lie-algebras-uni}, the Ricci tensor $\mathrm{Ric}(g)$ has signature $(3,0,0)$. 
Referring to Table \ref{table:simple-algebras-lorentzian-metrics}, write
\[e_1=\frac{1}{\sqrt{\mu_1}}x, \qquad e_2=\frac{1}{\sqrt{\mu_2}}y, \qquad e_3=\frac{1}{\sqrt{\mu_3}}z.\] 
Then $\{e_1,e_2,e_3\}$ is an orthonormal basis of $(\mathfrak{su}(2),\textup{(su)})$ satisfying $\varepsilon_1=\varepsilon_2=-\varepsilon_3=1$ and 
\[[e_1,e_2]=\frac{2\mu_3}{\sqrt{\mu_1\mu_2\mu_3}}e_3,\qquad[e_2,e_3]=\frac{2\mu_1}{\sqrt{\mu_1\mu_2\mu_3}}e_1,\qquad[e_3,e_1]=\frac{2\mu_2}{\sqrt{\mu_1\mu_2\mu_3}}e_2.\]
Comparing with \eqref{eq:Lie_brackets_Lor_unimodular} we get $a_{12}=a_{13}=b_{12}=0$, $a_{23}=-2\mu_1/\sqrt{\mu_1\mu_2\mu_3}$, $b_{13}=2\mu_2/\sqrt{\mu_1\mu_2\mu_3}$, and $\mu_3=\mu_1+\mu_2$. In particular, $\mu_1<\mu_2+\mu_3$.
\end{proof}
Hereafter we apply the same process to all other three-dimensional unimodular Lie algebras, checking which metrics carry left-invariant harmonic spinors.
As in the proof of Proposition \ref{prop:su(2)-harm-sp}, we use Table \ref{table:three-dim-lie-algebras-uni} and Tables \ref{table:simple-algebras-lorentzian-metrics}--\ref{table:non-simple-algebras-lorentzian-metrics} for useful metric and curvature data.
We only discuss those few cases admitting left-invariant harmonic spinors.
We set relevant bases up and leave computations and details to the reader.

\begin{proposition}
Suppose that $g$ is a Lorentzian metric on $\mathfrak g=\mathfrak{sl}(2,\mathbb R)$ carrying left-invariant harmonic spinors. 
Then $g$ is equivalent to one of the following metrics.
\begin{enumerate}
\item $\textup{(sll2)}$, with $\mu_1=\mu_3-\mu_2$ and $\mu_3>\mu_2>0$.
\item $\textup{(sll4)}$, with $2\beta^2+a^2\alpha=0$, $\beta>0$, and $-a^2<\alpha<0$.
\item $\textup{(sll6)}$, with $a=-2b$ and $b\neq0$.
\end{enumerate}
\end{proposition}
\begin{proof}
By Lemma \ref{lemma:Ric=-B_Lorentzian_uni} and Table \ref{table:three-dim-lie-algebras-uni}, the Ricci tensor $\mathrm{Ric}(g)$ has signature $(1,2,0)$.
This gives a first constraint on which metrics in Table \ref{table:simple-algebras-lorentzian-metrics} are admissible.
One now chooses orthonormal bases for each Lorentzian metric.
Computing the corresponding Lie brackets, one checks whether each Lie algebra admits left-invariant harmonic spinors, and if so, under which conditions on the parameters.
For $(\textup{sll2})$, write
\[e_1=\frac{1}{\sqrt{\mu_1}}x, \qquad e_2=-\frac{1}{\sqrt{\mu_3}}z, \qquad e_3=\frac{1}{\sqrt{\mu_2}}x.\]
For $(\textup{sll4})$, write
\begin{gather*}
e_1=\frac{1}{\sqrt{-KN}}x, \quad e_2=\frac{1}{|a|\sqrt{\alpha K/N}}y, \\
e_3=\frac{1}{\sqrt{K(\alpha^2-2\beta^2)/N}}\left(\frac{\beta}{N}x+z\right).
\end{gather*}
For $(\textup{sll6})$, we first define
\[
\lambda_{\pm} \coloneqq K(a^2M+2a^2(8\pm\sqrt{a^2+64})+(8\pm\sqrt{a^2+64})^2N),\]
then write $e_1=|b|z/2$ and
\[
e_2=\frac{1}{\sqrt{|\lambda_+|}}(ax-(8+\sqrt{a^2+64})y), \qquad e_3=\frac{1}{\sqrt{|\lambda_-|}}(ax-(8-\sqrt{a^2+64})y).
\]
In all cases, these bases are orthonormal and satisfy $\varepsilon_1=\varepsilon_2=-\varepsilon_3=1$.
Computing the Lie brackets $[e_i,e_j]$ and comparing with \eqref{eq:Lie_brackets_Lor_unimodular}, the claimed algebraic constraints on the parameters follow.
One rules out all other cases with the same method.
\end{proof}

\begin{proposition}
Suppose that $g$ is a Lorentzian metric on $\mathfrak g=\mathfrak{e}(2)$ carrying left-invariant harmonic spinors. 
Then $g$ is equivalent to the metric $\textup{(ee2)}$ with $u=-v<0$.
\end{proposition}
\begin{proof}
By Lemma \ref{lemma:Ric=-B_Lorentzian_uni} and Table \ref{table:three-dim-lie-algebras-uni}, the Ricci tensor $\mathrm{Ric}(g)$ has signature $(1,0,2)$. 
By the results in Table \ref{table:non-simple-algebras-lorentzian-metrics}, the only admissible metric is $(\textup{ee2})$ with $u=-v$.
Take the orthonormal basis
\[e_1= \frac{1}{\sqrt{v}}z, \qquad e_2=\frac{1}{\sqrt{v}}y-\sqrt{v}x, \qquad e_3=\frac{1}{\sqrt{v}}y.\]
One computes the structure constants by \eqref{eq:Lie_brackets_Lor_unimodular} with no contradiction, and the result follows.
\end{proof}

\begin{proposition}
Suppose that $g$ is a Lorentzian metric on $\mathfrak g=\mathfrak{e}(1,1)$ carrying left-invariant harmonic spinors. 
Then $g$ is equivalent to one of the following metrics.
\begin{enumerate}
\item $\textup{(sol1)}$, with $u=0$ and $v>0$.
\item $\textup{(sol4)}$, with $u>0$.
\item $\textup{(sol7)}$.
\end{enumerate}
\end{proposition}
\begin{proof}
By Lemma \ref{lemma:Ric=-B_Lorentzian_uni} and Table \ref{table:three-dim-lie-algebras-uni}, the Ricci tensor $\mathrm{Ric}(g)$ has signature $(0,1,2)$. 
The only admissible metrics from Table \ref{table:non-simple-algebras-lorentzian-metrics} are $(\textup{sol1})$, $(\textup{sol4})$, $(\textup{sol6})$, and $(\textup{sol7})$.
For the metric $(\textup{sol1})$ with $u=0$ and $v>0$, take the basis 
\[e_1=z, \qquad e_2=y, \qquad e_3= \frac{v}{2}x.\]
For $(\textup{sol4})$ with $u>0$, take the basis 
\[e_1= \sqrt{u}x, \qquad e_2=z, \qquad e_3=y.\]
For $(\textup{sol6})$ with $u>0$, put 
\[e_1=\frac{1}{u}x, \qquad e_2=\frac{\sqrt{2}}{2}\left(y+\left(1-\frac{u}{2}\right)z\right), \qquad e_3=\frac{\sqrt{2}}{2}\left(-y+\left(1+\frac{u}{2}\right)z\right).\]
Finally, for $(\textup{sol7})$ define
\[e_1=\frac{\sqrt{2}}{2}(x+z), \qquad e_2=y, \qquad e_3= \frac{\sqrt{2}}{2}(x-z).\]
By the usual procedure, one only rules out $(\textup{sol6})$.
\end{proof}
\begin{proposition}
Suppose that $g$ is a Lorentzian metric on $\mathfrak g=\mathfrak{heis}_3$ carrying left-invariant harmonic spinors. 
Then $g$ is equivalent to the metric $\textup{(nil0)}$.
\end{proposition}
\begin{proof}
By Lemma \ref{lemma:Ric=-B_Lorentzian_uni} and Table \ref{table:three-dim-lie-algebras-uni}, the Ricci tensor $\mathrm{Ric}(g)$ has signature $(0,0,3)$.
Then the only admissible metric from Table \ref{table:non-simple-algebras-lorentzian-metrics} is $(\textup{nil0})$.
Take the orthonormal basis
\[e_1=x, \qquad e_2=\frac{\sqrt{2}}{2}(y+z), \qquad e_3=\frac{\sqrt{2}}{2}(y-z).\]
The usual procedure yields no contradiction, and we conclude.
\end{proof}
We summarise all the above partial results into the following.
\begin{theorem}
\label{thm:3d_uni_Lorentzian_classification}
Let $(\mathfrak g,g)$ be a unimodular three-dimensional Lorentzian non-Abelian Lie algebra admitting left-invariant harmonic spinors.
Then $(\mathfrak g,g)$ is one of the cases in \emph{Table \ref{table:Lorentzian_uni}}.
\end{theorem}

For completeness, we check that applying Proposition \ref{prop:Dirac_almost_Abelian} and Proposition \ref{prop:Dirac_almost_Abelian_isotropic} to the almost Abelian Lorentzian Lie algebras yields the same results. 
We distinguish between the isotropic and non-isotropic case. 

\begin{proposition}
\label{prop:3d_UAA_non_isotropic}
Let $(\mathfrak g=\mathfrak h\rtimes_D\mathbb R,g)$ be a unimodular three-dimensional non-isotropic almost Abelian Lorentzian Lie algebra. Then it admits left-invariant harmonic spinors if and only if one of the following cases holds. 
\begin{enumerate}
\item $\mathfrak g\cong\mathfrak{e}(1,1)$, and $g$ is either equivalent to $\textup{(sol1)}$ with $u=0$ and $v>0$, or to $\textup{(sol4)}$ with $u>0$.
\item $\mathfrak g\cong\mathfrak{e}(2)$, and $g$ is equivalent to $\textup{(ee2)}$ with $u=-v<0$.
\item $\mathfrak g\cong\mathfrak{heis}_3$, and $g$ is equivalent to $\textup{(nil0)}$.
\end{enumerate}
\end{proposition}
\begin{proof}
Let $\{e_1,e_2,e_3\}$ be an orthonormal basis of $(\mathfrak g,g)$ satisfying $\varepsilon_1=\varepsilon_2=-\varepsilon_3=1$, $\mathfrak h=\mathrm{Span}_{\mathbb R}\{e_1,e_2\}$, and 
\[D
=\mathrm{ad}(e_3)|_{\mathfrak h}=
\begin{pmatrix}
d_{11}&d_{12}\\
d_{21}&-d_{11}
\end{pmatrix}
\] 
for $d_{ij}\in\mathbb R$.
Note that $\mathrm{tr}(D)=0$ since $\mathfrak g$ is unimodular. 
Let $\psi\in\Sigma$ be a non-zero left-invariant spinor. 
By Proposition \ref{prop:Dirac_almost_Abelian} we have 
\[
\slashed{D}\psi=
-\frac{1}{4}(d_{12}-d_{21})(e^1\wedge e^2\wedge e^3)\psi=\frac{i}{4}(d_{12}-d_{21})\psi=0
\] 
if and only $d_{21}=d_{12}$. 
That is, the endomorphism $D$ is equivalent to
\[\sqrt{d_{11}^2+d_{12}^2}
\begin{pmatrix}
-1&0\\
0&1
\end{pmatrix},
\] 
which implies that $\mathfrak g$ is completely solvable, therefore $\mathfrak g\cong\mathfrak{e}(1,1)$. 
A computation shows that in this situation we have $\sigma(\mathrm{Ric}(g))=(0,1,2)$. 
Set $t\coloneqq\sqrt{d_{11}^2+d_{12}^2}>0$ and choose $\theta\in\mathbb R$ such that $\cos(2\theta)=d_{11}/t$ and $\sin(2\theta)=d_{12}/t$. 
Take 
\[x=\frac{1}{t}e_3,\qquad y=\cos(\theta)e_1+\sin(\theta)e_2,\qquad z=-\sin(\theta)e_1+\cos(\theta)e_2.\]
The basis $\{x,y,z\}$ satisfies $[x,y]=y$, $[z,x]=z$, and $[z,y]=0$. 
The matrix of the metric $g$ in this new basis is 
\[
\begin{pmatrix}
-1/t^2&0&0\\
0&1&0\\
0&0&1
\end{pmatrix},
\] which gives the metric $\textup{(sol1)}$ with $u=0$ and $v=2t$. 
Note that in this case the metric $g$ restricted to $\mathfrak h$ is positive-definite.
    
We also have to consider the case where the metric $g$ restricted to the Abelian ideal $\mathfrak h$ is Lorentzian, that is, the case where $\{e_1,e_2,e_3\}$ is an orthonormal basis of $(\mathfrak g,g)$ satisfying $\varepsilon_1=-\varepsilon_2=\varepsilon_3=1$. Note that, in order to be compatible with the signature, the representation of $\mathrm{Cl}(2,1)$ has to be chosen as follows: 
\[
e^1=
\begin{pmatrix}
0&i\\
i&0
\end{pmatrix},
\qquad 
e^2=
\begin{pmatrix}
1&0\\
0&-1
\end{pmatrix},
\qquad 
e^3=\begin{pmatrix}
0&1\\
-1&0
\end{pmatrix}.
\]
In this situation, by Proposition \ref{prop:Dirac_almost_Abelian} we have 
\[\slashed{D}\psi=-\frac{1}{4}(d_{12}+d_{21})(e^1\wedge e^2\wedge e^3)\psi=-\frac{i}{4}(d_{12}+d_{21})\psi=0\] 
if and only if $d_{21}=-d_{12}$. 
That is, the endomorphism $D$ is equivalent to 
\[
\sqrt{d_{11}^2-d_{12}^2}
\begin{pmatrix}
-1&0\\
0&1
\end{pmatrix}.
\]
Set $t\coloneqq\sqrt{d_{11}^2-d_{12}^2}$. 
A computation shows that in this situation we have
\[\mathrm{Ric}(g)=
\begin{pmatrix}
0&0&0\\
0&0&0\\
0&0&-2t^2
\end{pmatrix}.
\]
We need to consider three different cases.

\textbf{Case $d_{11}^2-d_{12}^2>0$.} 
In this case $\mathfrak g\cong\mathfrak{e}(1,1)$ since $\mathfrak g$ is completely solvable, and $\sigma(\mathrm{Ric}(g))=(0,1,2)$. 
Choose $\theta\in\mathbb R$ such that $\cosh(2\theta)=d_{11}/t$ and $\sinh(\theta)=d_{12}/t$. 
Take 
\[x=-\frac{1}{t}e_3,\qquad y=\sinh(\theta)e_1+\cosh(\theta)e_2,\qquad z=\cosh(\theta)e_1+\sinh(\theta)e_2.\]
The basis $\{x,y,z\}$ satisfies $[x,y]=y$, $[z,x]=z$, and $[z,y]=0$. 
The matrix of the metric $g$ in this new basis is 
\[
\begin{pmatrix}
1/t^2&0&0\\
0&-1&0\\
0&0&1
\end{pmatrix},
\] 
which gives the metric $\textup{(sol4)}$ with $u=t^2$.

\textbf{Case $d_{11}^2-d_{12}^2<0$.} 
In this case $\mathfrak g\cong\mathfrak{e}(2)$ since $\mathfrak g$ is not completely solvable, and $\sigma(\mathrm{Ric}(g))=(1,0,2)$. 
Therefore, the metric $g$ is equivalent to $\textup{(ee2)}$ with $u=-v<0$.

\textbf{Case $d_{11}^2-d_{12}^2=0$.} 
In this case $\mathfrak g\cong\mathfrak{heis}_3$ since $\mathfrak g$ is nilpotent, and $\sigma(\mathrm{Ric}(g))=(0,0,3)$. 
Therefore, the metric $g$ is equivalent to $\textup{(nil0)}$.
\end{proof}

\begin{proposition}
\label{prop:3d_UAA_isotropic}
Let $(\mathfrak g=\mathfrak h\rtimes_D\mathbb R,g)$ be a unimodular three-dimensional isotropic almost Abelian Lorentzian Lie algebra. 
Then it admits left-invariant harmonic spinors if and only if one of the following cases holds.
\begin{enumerate}
\item $\mathfrak g\cong\mathfrak{e}(1,1)$, and $g$ is equivalent to $\textup{(sol7)}$.
\item $\mathfrak g\cong\mathfrak{heis}_3$, and $g$ is equivalent to $\textup{(nil0)}$.
\end{enumerate}
\end{proposition}
\begin{proof}
Let $\{v_1,v_2,v_3\}$ be a basis of $\mathfrak g$ with $\mathfrak h=\mathrm{Span}_{\mathbb R}\{v_1,v_2\}$ and 
\[D=\mathrm{ad}(v_3)|_{\mathfrak h}=
\begin{pmatrix}
d_{11}&d_{12}\\
d_{21}&-d_{11}
\end{pmatrix}
\] for some $d_{ij}\in\mathbb R$. 
Note that $\mathrm{tr}(D)=0$ since $\mathfrak g$ is unimodular. 
The metric $g$ in this basis takes the form $g=v^1\otimes v^1+v^2\odot v^3$. 
Let 
\[e_1=v_1, \qquad e_2=\frac{\sqrt{2}}{2}(v_2+v_3), \qquad e_3=\frac{\sqrt{2}}{2}(v_2-v_3).\] 
Then $\{e_1,e_2,e_3\}$ is an orthonormal basis of $(\mathfrak g,g)$ with $\varepsilon_1=\varepsilon_2=-\varepsilon_3=1$. 
Let $\psi\in\Sigma$ be a non-zero left-invariant spinor. 
In this situation, by Proposition \ref{prop:Dirac_almost_Abelian_isotropic} we have
\[
\slashed{D}\psi=\frac{1}{4}d_{12}(e^1\wedge e^2\wedge e^3)\psi=-\frac{i}{4}d_{12}\psi=0\] 
if and only if $d_{12}=0$. 
Set $x_1\coloneqq d_{11}$ and $x_2\coloneqq d_{21}$. 
A computation shows that in this situation we have
\[\mathrm{Ric}(g)
=
\begin{pmatrix}
0&0&0\\
0&0&0\\
0&0&-2x_1^2
\end{pmatrix}
\] 
with respect to the basis $\{v_1,v_2,v_3\}$. 
We need to consider two different cases.

\textbf{Case $x_1\neq0$.} 
In this case $\mathfrak g\cong\mathfrak{e}(1,1)$ since $\mathfrak g$ is completely solvable, and $\sigma(\mathrm{Ric}(g))=(0,1,2)$. 
Take
\[x=\frac{1}{x_1}v_1-\frac{x_2}{2x_1^2}v_1-\frac{x_2^2}{8x_1^3}v_3,\qquad y=v_1+\frac{x_2}{2x_1}v_2,\qquad z=x_1v_2.\]
The basis $\{x,y,z\}$ satisfies $[x,y]=y$, $[z,x]=z$, and $[z,y]=0$. 
The matrix of the metric $g$ in this new basis is 
\[
\begin{pmatrix}
0&0&1\\
0&1&0\\
1&0&0
\end{pmatrix},
\] 
which is the metric $\textup{(sol4)}$.
    
\textbf{Case $x_1=0$.} 
In this case $\mathfrak g\cong\mathfrak{heis}_3$ since $\mathfrak g$ is nilpotent, and $\sigma(\mathrm{Ric}(g))=(0,0,3)$. 
Therefore, the metric $g$ is equivalent to $\textup{(nil0)}$.
\end{proof}

\subsection{Non-unimodular case}

Unlike in the previous case, in the non-unimodular one we do not have a priori linear restrictions on the coefficients $a_{ij},b_{ij},c_{ij}$, $1\leq i<j\leq3$.
Hence, we would have to find coefficients solving the quadratic equations \eqref{eq:d^2=0} and \eqref{eq:Dirac_Lorentzian}. 
However, every non-unimodular three-dimensional Lie algebra is almost Abelian, hence we can apply Proposition \ref{prop:Dirac_almost_Abelian} and Proposition \ref{prop:Dirac_almost_Abelian_isotropic} to classify those metrics carrying left-invariant harmonic spinors. 
We will proceed as in Propositions \ref{prop:3d_UAA_non_isotropic} and \ref{prop:3d_UAA_isotropic} and compare our results with the classifications obtained in \cite{ha-lee3}. 
In the arXiv version of the latter paper \cite{ha-lee2}, one can find useful tables with more information about the curvature of the Lorentzian metrics. 
We will use that information in the following proofs. 
We next state the main classification results.

\begin{theorem}[Ha--Lee \cite{ha-lee3}]
\label{thm:Ha_Lee_RH3}
Let $\mathbb R^2\rtimes_{\mathrm{Id}}\mathbb R$ be spanned by $\{x,y,z\}$ satisfying
\[[x,y]=0,\qquad [z,x]=x,\qquad [z,y]=y.\]
Any Lorentzian metric on $\mathbb R^2\rtimes_{\mathrm{Id}}\mathbb R$ is equivalent to a metric whose associated matrix with respect to the basis $\{x,y,z\}$ is of the form
\[
\begin{pmatrix}
1&0&0\\
0&-\epsilon&0\\
0&0&\epsilon\mu
\end{pmatrix},
\qquad
\begin{pmatrix}
1&0&0\\
0&0&1\\
0&1&0
\end{pmatrix},
\] where $\epsilon=\pm1$ and $\mu>0$.
\end{theorem}

\begin{theorem}[Ha--Lee \cite{ha-lee3}]
\label{thm:Ha_Lee_g(c)_c>1}
Let $c>1$ and let $\mathfrak g(c)$ be spanned by $\{x,y,z\}$ satisfying
\[[x,y]=0,\qquad [z,x]=y,\qquad [z,y]=-cx+2y.\]
Any Lorentzian metric on $\mathfrak g(c)$ is equivalent to a metric whose associated matrix with respect to the basis $\{x,y,z\}$ is of the form 
\[
\begin{pmatrix}
\mu&0&0\\
0&0&1\\
0&1&0
\end{pmatrix},
\qquad
\begin{pmatrix}
1&1&0\\
1&\tau&0\\
0&0&\mu
\end{pmatrix},
\qquad
\begin{pmatrix}
1&1&0\\
1&\nu&0\\
0&0&-\mu
\end{pmatrix},
\] 
where $\mu>0$, $\tau<1$, and $1<\nu\leq c$.
\end{theorem}

\begin{theorem}[Ha--Lee \cite{ha-lee3}]
\label{thm:Ha_Lee_g(1)}
Let $\mathfrak g(1)$ be spanned by $\{x,y,z\}$ satisfying
\[[x,y]=0,\qquad [z,x]=x,\qquad [z,y]=x+y.\]
Any Lorentzian metric on $\mathfrak g(1)$ is equivalent to a metric whose associated matrix with respect to the basis $\{x,y,z\}$ is of the form
\begin{alignat*}{3}
&
\begin{pmatrix}
0&0&1\\
0&\mu&0\\
1&0&0
\end{pmatrix},
\qquad &&
\begin{pmatrix}
\mu&0&0\\
0&0&1\\
0&1&0
\end{pmatrix},
\qquad &&
\begin{pmatrix}
1&0&0\\
0&-\nu&0\\
0&0&\mu
\end{pmatrix},
\\
&
\begin{pmatrix}
1&0&0\\
0&\nu&0\\
0&0&-\mu
\end{pmatrix},\qquad &&
\begin{pmatrix}
-1&0&0\\
0&\nu&0\\
0&0&\mu
\end{pmatrix},
\qquad &&
\begin{pmatrix}
0&\epsilon&0\\
\epsilon&0&0\\
0&0&\mu
\end{pmatrix},
\end{alignat*} 
where $\epsilon=\pm 1$, $\mu>0$ and $\nu>0$.
\end{theorem}

\begin{theorem}[Ha--Lee \cite{ha-lee3}]
\label{thm:Ha_Lee_g(c)_c<1}
Let $c<1$ and let $\mathfrak g(c)$ be spanned by $\{x,y,z\}$ satisfying
\[[x,y]=0,\qquad [z,x]=(1+w)x,\qquad [z,y]=(1-w)y,\]
where $w=\sqrt{1-c}$. Any Lorentzian metric on $\mathfrak g(c)$ is equivalent to a metric whose associated matrix with respect to the basis $\{x,y,z\}$ is of the form
\begin{alignat*}{3}
&
\begin{pmatrix}
0&0&1\\
0&1&0\\
1&0&0
\end{pmatrix},
\qquad &&
\begin{pmatrix}
1&0&0\\
0&0&1\\
0&1&0
\end{pmatrix},
\qquad &&
\begin{pmatrix}
1&1&0\\
1&1&\mu\\
0&\mu&0
\end{pmatrix},
\\
&\begin{pmatrix}
1&0&0\\
0&1&0\\
0&0&-\mu
\end{pmatrix},\qquad &&
\begin{pmatrix}
\epsilon&0&0\\
0&-\epsilon&0\\
0&0&\mu
\end{pmatrix},
\qquad &&
\begin{pmatrix}
0&1&0\\
1&0&0\\
0&0&\mu
\end{pmatrix},
\\
&
\begin{pmatrix}
0&1&0\\
1&\epsilon&0\\
0&0&\mu
\end{pmatrix},
\qquad &&
\begin{pmatrix}
1&1&0\\
1&\tau&0\\
0&0&\nu
\end{pmatrix},
\qquad &&
\begin{pmatrix}
-1&1&0\\
1&-\eta&0\\
0&0&\mu
\end{pmatrix},
\end{alignat*} 
where $\epsilon=\pm1$, $\mu>0$, $\nu(\tau-1)<0$, and $\eta<1$.
\end{theorem}

We now proceed with the classification of the metrics carrying left-invariant harmonic spinors. 
First we start with the non-isotropic case, where the codimension one Abelian ideal $\mathfrak h$ is non-degenerate with respect to the metric $g$. We distinguish the cases where $\mathfrak h$ is Riemannian or Lorentzian.

\begin{lemma}
Let $(\mathfrak g=\mathfrak h\rtimes_D\mathbb R,g)$ be a non-unimodular three-dimensional non-isotropic almost Abelian Lorentzian Lie algebra with $\mathfrak h$ Riemannian. 
Then it does not admit left-invariant harmonic spinors.
\end{lemma}
\begin{proof}
Let $\{e_1,e_2,e_3\}$ be an orthonormal basis of $(\mathfrak g,g)$ satisfying $\varepsilon_1=\varepsilon_2=-\varepsilon_3=1$, $\mathfrak h=\mathrm{Span}_{\mathbb R}\{e_1,e_2\}$, and 
\[D=\mathrm{ad}(e_3)|_{\mathfrak h}=
\begin{pmatrix}
d_{11}&d_{12}\\
d_{21}&d_{22}
\end{pmatrix}
\]
for some $d_{ij}\in\mathbb R$. 
Note that $\mathrm{tr}(D)=d_{11}+d_{22}\neq0$ since $\mathfrak g$ is non-unimodular. 
Let $\psi\in\Sigma$ be a left-invariant spinor. 
Recall that we use the representation of $\mathrm{Cl}(2,1)$ given by \eqref{eq:rep_Cl(2,1)}. 
By Proposition \ref{prop:Dirac_almost_Abelian} we have 
\[
\slashed{D}\psi=\begin{pmatrix}
-\frac{1}{2}(d_{11}+d_{22})+\frac{i}{4}(d_{12}-d_{21})&0\\
0&\frac{1}{2}(d_{11}+d_{22})+\frac{i}{4}(d_{12}-d_{21})
\end{pmatrix}\psi.
\]
Left-invariant harmonic spinors exist if and only if 
\[\det(\slashed{D})=-\frac{1}{4}(d_{11}+d_{22})^2-\frac{1}{16}(d_{12}-d_{21})^2=0,\] 
so in particular $\mathrm{tr}(D)=d_{11}+d_{22}=0$, which is impossible as $\mathrm{tr}(D) \neq 0$.
\end{proof}

We then consider the case where the metric $g$ restricted to the Abelian ideal $\mathfrak h$ is Lorentzian, that is, the case where $\{e_1,e_2,e_3\}$ is an orthonormal basis of $(\mathfrak g,g)$ satisfying $\varepsilon_1=-\varepsilon_2=\varepsilon_3=1$. 
Note that, in order to be compatible with the signature, the representation of $\mathrm{Cl}(2,1)$ here has to be chosen as follows: 
\[
e^1=
\begin{pmatrix}
0&i\\
i&0
\end{pmatrix},
\qquad 
e^2=
\begin{pmatrix}
1&0\\
0&-1
\end{pmatrix},
\qquad 
e^3=\begin{pmatrix}
0&1\\
-1&0
\end{pmatrix}.
\]
In this situation, by Proposition \ref{prop:Dirac_almost_Abelian} we have 
\[
\slashed{D}\psi=\begin{pmatrix}
-\frac{i}{4}(d_{12}+d_{21})&-\frac{1}{2}(d_{11}+d_{22})\\
\frac{1}{2}(d_{11}+d_{22})&-\frac{i}{4}(d_{12}+d_{21})
\end{pmatrix}\psi.
\]
Left-invariant harmonic spinors exist if and only if 
\[\det(\slashed{D})=\frac{1}{4}(d_{11}+d_{22})^2-\frac{1}{16}(d_{12}+d_{21})^2=0,\] which amounts to 
\[d_{12}+d_{21}=2\sigma(d_{11}+d_{22}), \qquad \sigma \in \{\pm 1\}.\]

\begin{lemma}
Let $(\mathfrak g=\mathfrak h\rtimes_D\mathbb R,g)$ be a non-unimodular three-dimensional non-isotropic almost Abelian Lorentzian Lie algebra with $\mathfrak h$ Lorentzian. 
If it admits left-invariant harmonic spinors, then $\mathfrak g\not\cong\mathbb R^2\rtimes_{\mathrm{Id}}\mathbb R$.
\end{lemma}
\begin{proof}
By \cite[Prop.\ 1]{freibert}, the almost Abelian Lie algebra $\mathfrak g=\mathbb R^2\rtimes_D\mathbb R$ is isomorphic to $\mathbb R^2\rtimes_{\mathrm{Id}}\mathbb R$ if and only if there exists $\alpha \in\mathbb R\setminus\{0\}$ and $P\in\mathrm{GL}(2,\mathbb R)$ such that $D=\alpha P^{-1}\mathrm{Id} P = \alpha \mathrm{Id}$. 
This implies $d_{12}=d_{21}=0$ and $d_{11}=d_{22}=\alpha $. 
Then $d_{12}+d_{21}=2\sigma(d_{11}+d_{22})$ forces $\alpha=0$, contradiction. 
\end{proof}

Therefore, $\mathfrak g\cong\mathfrak g(c)$ for some $c\in\mathbb R$. 
We distinguish three cases depending on the eigenvalues of $D$, which are given by 
\[\lambda_\pm=\frac{1}{2}\left(\mathrm{tr}(D)\pm\sqrt{\mathrm{tr}(D)^2-4\det(D)}\right),\] 
with $\mathrm{tr}(D)\neq0$ and $d_{12}+d_{21}=2\sigma\mathrm{tr}(D)$. 
We also note that in this situation the scalar curvature of the metric $g$ is given by 
\[\mathrm{Scal}(g)=2\det(D).\]

\begin{lemma}
Let $(\mathfrak g=\mathfrak h\rtimes_D\mathbb R,g)$ be a non-unimodular three-dimensional non-isotropic almost Abelian Lorentzian Lie algebra with $\mathfrak h$ Lorentzian and $\mathrm{tr}(D)^2=4\det(D)$. 
If it admits left-invariant harmonic spinors, then $\mathfrak g\cong\mathfrak g(1)$ and $g$ is equivalent to
\[\begin{pmatrix}
\epsilon&0&0\\
0&-\epsilon/16&0\\
0&0&\mu
\end{pmatrix},
\qquad
\epsilon=\pm1,\quad \mu>0.
\]
\end{lemma}
\begin{proof}
In this case $D$ has two equal non-zero eigenvalues, but is not diagonalizable. 
Hence $\mathfrak g$ is isomorphic to $\mathfrak g(1)$. 
Moreover, the scalar curvature of $g$ is positive since $\mathrm{Scal}(g)=\frac{1}{2}\mathrm{tr}(D)^2$. 
Let $\lambda=\frac{1}{2}\mathrm{tr}(D)$ and consider 
\[D-\lambda\mathrm{Id}=
\begin{pmatrix}
\frac{1}{2}(d_{11}-d_{22})&d_{12}\\
d_{21}&-\frac{1}{2}(d_{11}-d_{22})
\end{pmatrix}.
\]
Suppose that $(d_{11}-d_{22},d_{21})\neq(0,0)$ and consider the basis $\{x,y,z\}$ given by 
\[x=\frac{1}{2}(d_{11}-d_{22})e_1+d_{21}e_2,\qquad y=\frac{\mathrm{tr}(D)}{2}e_1,\qquad z=\frac{2}{\mathrm{tr}(D)}e_3.\]
The basis $\{x,y,z\}$ satisfies $[x,y]=0$, $[z,x]=x$, and $[z,y]=x+y$, and the matrix of the metric $g$ in this basis is 
\[g=
\begin{pmatrix}
g_{11}&g_{12}&0\\
g_{12}&g_{22}&0\\
0&0&g_{33}
\end{pmatrix}
\coloneqq
\begin{pmatrix}
\frac{1}{4}(d_{11}-d_{22})^2-d_{21}^2&\frac{1}{4}(d_{11}-d_{22})\mathrm{tr}(D)&0\\
\frac{1}{4}(d_{11}-d_{22})\mathrm{tr}(D)&\frac{\mathrm{tr}(D)^2}{4}&0\\
0&0&\frac{4}{\mathrm{tr}(D)^2}
\end{pmatrix}.
\]
We note that $g_{11}=-2\sigma d_{21}\mathrm{tr}(D)\neq0$ and that the restriction of $g$ to $\mathrm{Span}_{\mathbb R}\{x,y\}$ is Lorentzian. 
Combining this with the fact that $\mathrm{Scal}(g)>0$, we conclude that only two of the (families of) metrics on $\mathfrak g(1)$ from Theorem \ref{thm:Ha_Lee_g(1)} could admit left-invariant harmonic spinors. 
We now identify them.
Set 
\[A=\begin{pmatrix}
1/\sqrt{|g_{11}|}&-g_{12}/g_{11}\sqrt{|g_{11}|}&0\\
0&1/\sqrt{|g_{11}|}&0\\
0&0&1
\end{pmatrix}\in\mathrm{Aut}(\mathfrak g(1)),\]
then for $\epsilon = \mathrm{sgn}(g_{11}) = \pm1$ we have
\[A^TgA=
\begin{pmatrix}
\epsilon &0&0\\
0&-\epsilon\nu&0\\
0&0&\mu
\end{pmatrix},
\quad
\nu=\frac{g_{12}^2-g_{11}g_{22}}{g_{11}^2}=\frac{1}{16},\quad\mu=g_{33}=\frac{4}{\mathrm{tr}(D)^2}>0.
\]
Suppose that $(d_{11}-d_{22},d_{21})=(0,0)$, so 
$D=d_{11}
\left(
\begin{smallmatrix}
1&4\sigma\\
0&1
\end{smallmatrix}
\right)$, 
and consider the basis $\{x,y,z\}$ given by 
\[x=4\sigma d_{11}e_1,\qquad y=d_{11}e_2,\qquad z=\frac{1}{d_{11}}e_3.\]
This satisfies $[x,y]=0$, $[z,x]=x$, and $[z,y]=x+y$, and the matrix of $g$ in this basis is $g=\mathrm{diag}(16d_{11}^2,-d_{11}^2,1/d_{11}^2).$
Choosing $A=\mathrm{diag}\left(1/4d_{11},1/4d_{11},1\right)\in\mathrm{Aut}(\mathfrak g(1))$ we conclude that 
\[A^TgA=
\begin{pmatrix}
1&0&0\\
0&-\nu&0\\
0&0&\mu
\end{pmatrix},
\qquad 
\nu=\frac{1}{16},\quad\mu=\frac{1}{d_{11}^2}>0,
\] which coincides with the former of the above metrics.
\end{proof}

\begin{lemma}
Let $(\mathfrak g=\mathfrak h\rtimes_D\mathbb R,g)$ be a non-unimodular three-dimensional non-isotropic almost Abelian Lorentzian Lie algebra with $\mathfrak h$ Lorentzian and $\mathrm{tr}(D)^2<4\det(D)$. 
If it admits left-invariant harmonic spinors, then $\mathfrak g\cong\mathfrak g(c)$ with $c>1$ and $g$ is equivalent to 
\[\begin{pmatrix}
1&1&0\\
1&\tau&0\\
0&0&\mu
\end{pmatrix},\qquad\tau=-(c+6)\pm4\sqrt{c+3}<1,\quad\mu>0.
\]
\end{lemma}
\begin{proof}
In this case $D$ has two different complex non-real eigenvalues. 
Hence $\mathfrak g$ is isomorphic to $\mathfrak g(c)$ with $c>1$. 
Moreover, the scalar curvature of $g$ is positive since $\mathrm{Scal}(g)>\frac{1}{2}\mathrm{tr}(D)^2$. 
Consider the basis $\{x,y,z\}$ given by 
\[x=e_1,\qquad y=\frac{2}{\mathrm{tr}(D)}(d_{11}e_1+d_{21}e_2),\qquad z=\frac{2}{\mathrm{tr}(D)}e_3.\]
The basis $\{x,y,z\}$ satisfies $[x,y]=0$, $[z,x]=y$, and $[z,y]=-cx+2y$, where $c=4\det(D)/\mathrm{tr}(D)^2>1$. 
The matrix of the metric $g$ in this basis is \[g=
\begin{pmatrix}
1&2d_{11}/\mathrm{tr}(D)&0\\
2d_{11}/\mathrm{tr}(D)&4(d_{11}^2-d_{21}^2)/\mathrm{tr}(D)^2&0\\
0&0&4/\mathrm{tr}(D)^2
\end{pmatrix}.
\]
Since the restriction of $g$ to $\mathrm{Span}_{\mathbb R}\{x,y\}$ is Lorentzian, we conclude that only the second of the (families of) metrics on $\mathfrak g(c)$ with $c>1$ from Theorem \ref{thm:Ha_Lee_g(c)_c>1} could admit left-invariant harmonic spinors. Note that the scalar curvature is given by 
\[\mathrm{Scal}(g)=2\det(D)=\frac{c}{2}\mathrm{tr}(D)^2=\frac{2c}{\mu},\] 
where $\mu=4/\mathrm{tr}(D)^2>0$. 
Comparing this with the scalar curvature of the second metric in Theorem \ref{thm:Ha_Lee_g(c)_c>1}, which is 
\[\frac{\tau^2-2(c-6)\tau+(c^2-12)}{2\mu(1-\tau)},\] 
we conclude that $\tau=-(c+6)\pm4\sqrt{c+3}$, which satisfies $\tau<1$ for all $c>1$, as expected.
\end{proof}

When $\mathrm{tr}(D)^2>4\det(D)$, we may have $\det(D)=0$. We address this case first.
\begin{lemma}
Let $(\mathfrak g=\mathfrak h\rtimes_D\mathbb R,g)$ be a non-unimodular three-dimensional non-isotropic almost Abelian Lorentzian Lie algebra with $\mathfrak h$ Lorentzian and $\det(D)=0$. 
If it admits left-invariant harmonic spinors, then $\mathfrak g\cong\mathfrak g(0)$ and $g$ is equivalent to
\[
\begin{pmatrix}
\epsilon&1&0\\
1&3\epsilon/4&0\\
0&0&\nu
\end{pmatrix}, 
\qquad
\epsilon = \pm1, \quad
\nu>0.
\]
\end{lemma}
\begin{proof}
In this case $D$ has a zero and a non-zero eigenvalue. 
Hence, $\mathfrak g$ is isomorphic to $\mathfrak g(0)$. 
Moreover, the scalar curvature of $g$ is zero since $\mathrm{Scal}(g)=2\det(D)=0$.

Suppose that $d_{12}\neq0$ and consider the basis $\{x,y,z\}$ given by 
\[x=d_{12}e_1+d_{22}e_2,\qquad y=-d_{12}e_1+d_{11}e_2,\qquad z=\frac{2}{\mathrm{tr}(D)}e_3.\]
The basis satisfies $[x,y]=0$, $[z,x]=2x$, and $[z,y]=0$, and the matrix of $g$ in this basis is 
\[g=
\begin{pmatrix}
g_{11}&g_{12}&0\\
g_{12}&g_{22}&0\\
0&0&g_{33}
\end{pmatrix}
\coloneqq
\begin{pmatrix}
d_{12}^2-d_{22}^2&-d_{12}^2-d_{11}d_{22}&0\\
-d_{12}^2-d_{11}d_{22}&d_{12}^2-d_{11}^2&0\\
0&0&4/\mathrm{tr}(D)^2
\end{pmatrix}.
\]
Since the restriction of $g$ to $\mathrm{Span}_{\mathbb R}\{x,y\}$ is Lorentzian, and $\mathrm{Scal}(g)=0$, we conclude that only two of the (families of) metrics on $\mathfrak g(0)$ from Theorem \ref{thm:Ha_Lee_g(c)_c<1} could admit left-invariant harmonic spinors. 
We now identify them.
Set
\[A=\mathrm{diag}\left(\frac{1}{\sqrt{|g_{11}|}},\frac{\sqrt{|g_{11}|}}{g_{12}},1\right)\in\mathrm{Aut}(\mathfrak g(0)),\]
then for $\epsilon = \mathrm{sgn}(g_{11}) = \pm 1$ we have
\[
A^TgA=
\begin{pmatrix}
\epsilon &1&0\\
1&\epsilon\tau&0\\
0&0&\nu
\end{pmatrix},
\qquad
\tau=\frac{g_{11}g_{22}}{g_{12}^2}=\frac{3}{4},\quad\nu=g_{33}=\frac{4}{\mathrm{tr}(D)^2}>0.
\]
Suppose that $d_{12}=0$ and consider the basis $\{x,y,z\}$ given by 
\[x=d_{11}e_1+d_{21}e_2,\qquad y=-d_{22}e_1+d_{21}e_2,\qquad z=\frac{2}{\mathrm{tr}(D)}e_3.\]
The basis satisfies $[x,y]=0$, $[z,x]=2x$, and $[z,y]=0$, and the matrix of $g$ in this basis is 
\[g=
\begin{pmatrix}
d_{11}^2-4\mathrm{tr}(D)^2&-4\mathrm{tr}(D)^2&0\\
-4\mathrm{tr}(D)^2&d_{22}^2-4\mathrm{tr}(D)^2&0\\
0&0&4/\mathrm{tr}(D)^2
\end{pmatrix},
\] 
where we have used $d_{11}d_{22}=0$ and $d_{21}=2\sigma\mathrm{tr}(D)$. 
In fact, since $d_{11}d_{22}=0$, either $d_{11}=0$ or $d_{22}=0$. 
In both cases $g_{11}=d_{11}^2-4\mathrm{tr}(D)^2<0$. 
Note that in this case the restriction of $g$ to $\mathrm{Span}_{\mathbb R}\{x,y\}$ is also Lorentzian. 
Taking 
\[A=\mathrm{diag}\left(\frac{1}{\sqrt{-g_{11}}},\frac{\sqrt{-g_{11}}}{g_{12}},1\right)\in\mathrm{Aut}(\mathfrak g(0))\] we obtain 
\[A^TgA
=
\begin{pmatrix}
-1&1&0\\
1&-\tau&0\\
0&0&\nu
\end{pmatrix},
\qquad\tau=\frac{3}{4},\qquad\nu=\frac{4}{\mathrm{tr}(D)^2}>0,
\] which is one of the metrics above.
\end{proof}

\begin{lemma}
Let $(\mathfrak g=\mathfrak h\rtimes_D\mathbb R,g)$ be a non-unimodular three-dimensional non-isotropic almost Abelian Lorentzian Lie algebra with $\mathfrak h$ Lorentzian and $\mathrm{tr}(D)^2>4\det(D)\neq0$. 
If it admits left-invariant harmonic spinors, then: 
\begin{enumerate}
\item $\mathfrak g\cong\mathfrak g(-3)$ and $g$ is equivalent to one of the following metrics
\[
\begin{pmatrix}
0&1&0\\
1&\epsilon&0\\
0&0&\mu
\end{pmatrix},
\qquad \epsilon \in \{0,\pm 1\}, \quad
\mu>0.
\]
\item $\mathfrak g\cong\mathfrak g(c)$ with $c<1$ and $c\notin\{-3,0\}$ and $g$ is equivalent to one of the following metrics
\[
\begin{pmatrix}
\epsilon&1&0\\
1&\epsilon(c+3)/4&0\\
0&0&\mu
\end{pmatrix},
\qquad \epsilon = \pm1, \quad \mu>0.
\]
\end{enumerate}
\end{lemma}

\begin{proof}
In this case $D$ has two non-zero distinct real eigenvalues. 
Hence, $\mathfrak g$ is isomorphic to $\mathfrak g(c)$ with $c<1$ and $c\neq0$. 
Let 
\[t=\mathrm{tr}(D),\qquad\delta=\det(D),\qquad\lambda_{\pm}=\frac{1}{2}(t\pm\sqrt{t^2-4\delta}),\qquad w=\sqrt{1-\frac{4\delta}{t^2}}.\]
Suppose that $d_{12}\neq0$ and consider the basis $\{x,y,z\}$ given by 
\[x=d_{12}e_1+(\lambda_+-d_{11})e_2,\qquad y=d_{12}e_1+(\lambda_--d_{11})e_2,\qquad z=\frac{2}{t}e_3.\]
The basis satisfies $[x,y]=0$, $[z,x]=(1+w)x$, and $[z,y]=(1-w)y$, and the matrix of $g$ in this basis is 
\[
g=(g_{ij})
\coloneqq
\begin{pmatrix}
d_{12}^2-(\lambda_+-d_{11})^2&d_{12}^2-(\lambda_+-d_{11})(\lambda_--d_{11})&0\\
d_{12}^2-(\lambda_+-d_{11})(\lambda_--d_{11})&d_{12}^2-(\lambda_--d_{11})^2&0\\
0&0&4/t^2
\end{pmatrix}.
\]
We note that $g_{12}=2\sigma td_{12}\neq0$ and that the restriction of $g$ to $\mathrm{Span}_{\mathbb R}\{x,y\}$ is Lorentzian. 
Moreover, the scalar curvature of $g$ is given by 
\[\mathrm{Scal}(g)=2\det(D)=2\frac{t^2(1-w^2)}{4}=\frac{ct^2}{2},\] 
where $w=\sqrt{1-c}$. 
Next, we identify which metrics on $\mathfrak g\cong\mathfrak g(c)$ with $c<1$ and $c\neq0$ from Theorem~\ref{thm:Ha_Lee_g(c)_c<1} could admit left-invariant harmonic spinors. 
\begin{enumerate}
\item If $g_{11}=0$ and $g_{22}=0$ we take $A=\mathrm{diag}(1/g_{12},1,1)\in\mathrm{Aut}(\mathfrak g(c))$ and therefore 
\[A^TgA
=
\begin{pmatrix}
0&1&0\\
1&0&0\\
0&0&\mu
\end{pmatrix},
\qquad
\mu=g_{33}=\frac{4}{t^2}>0.
\] 
This metric has scalar curvature $-6/\mu$, which is equal to $ct^2/2=2c/\mu$ if and only if $c=-3$.
\item If $g_{11}=0$ and $g_{22}\neq 0$ we take $A=\mathrm{diag}(\sqrt{|g_{22}|}/g_{12},1/\sqrt{|g_{22}|},1)\in\mathrm{Aut}(\mathfrak g(c))$ and therefore 
\[A^TgA
=
\begin{pmatrix}
0&1&0\\
1&\epsilon &0\\
0&0&\mu
\end{pmatrix},
\qquad \epsilon = \mathrm{sgn}(g_{22}) = \pm 1, \quad \mu=g_{33}=\frac{4}{t^2}>0.
\] 
Again, the scalar curvature is $-6/\mu$, which is equal to $ct^2/2=2c/\mu$ if and only if $c=-3$. 
\item If $g_{11}\neq 0$ we take $A=\mathrm{diag}(1/\sqrt{|g_{11}|},\sqrt{|g_{11}|}/g_{12},1)\in\mathrm{Aut}(\mathfrak g(c))$ and therefore for $\epsilon = \mathrm{sgn}(g_{11}) = \pm1$ we have
\[A^TgA
=
\begin{pmatrix}
\epsilon&1&0\\
1&\epsilon\tau&0\\
0&0&\mu
\end{pmatrix},
\qquad\tau=\frac{g_{11}g_{22}}{g_{12}^2}=\frac{c+3}{4},\quad\mu=\frac{4}{t^2}>0.
\] 
For this value of $\tau$, this metric has scalar curvature $2c/\mu$.
\end{enumerate}
Suppose that $d_{12}=0$ and consider the basis $\{x,y,z\}$ given by 
\[x=(d_{11}-d_{22})e_1+d_{21}e_2,\qquad y=d_{21}e_2,\qquad z=\frac{2}{t}e_3.\]
The basis satisfies $[x,y]=0$, $[z,x]=(1+w)x$, and $[z,y]=(1-w)y$ for $w=(d_{11}-d_{22})/t$, and the matrix of $g$ in this basis is 
\[g=
\begin{pmatrix}
g_{11}&g_{12}&0\\
g_{12}&g_{22}&0\\
0&0&g_{33}
\end{pmatrix}
\coloneqq
\begin{pmatrix}
(d_{11}-d_{22})^2-d_{21}^2&-4t^2&0\\
-4t^2&-d_{21}^2&0\\
0&0&4/t^2
\end{pmatrix}.
\]
Arguing as before, we obtain the same five metrics as in the $d_{12}\neq0$ case.
\end{proof}

We conclude with the isotropic case, where the codimension one Abelian ideal $\mathfrak h$ is degenerate with respect to the metric $g$.
\begin{lemma}
Let $(\mathfrak g=\mathfrak h\rtimes_D\mathbb R,g)$ be a non-unimodular three-dimensional isotropic almost Abelian Lorentzian Lie algebra. 
Then it admits left-invariant harmonic spinors if and only if one of the following cases holds. 
\begin{enumerate}
\item $\mathfrak g\cong\mathbb R^2\rtimes_{\mathrm{Id}}\mathbb R$, and $g$ is equivalent to the metric 
\[
\begin{pmatrix}
1&0&0\\
0&0&1\\
0&1&0
\end{pmatrix}.
\]
\item $\mathfrak g\cong\mathfrak g(1)$, and $g$ is equivalent to the metric 
\[
\begin{pmatrix}
0&0&1\\
0&\mu&0\\
1&0&0
\end{pmatrix},\qquad\mu>0.
\]
\item $\mathfrak g\cong\mathfrak g(0)$, and $g$ is equivalent to one of the following metrics 
\[
\begin{pmatrix}
0&0&1\\
0&1&0\\
1&0&0
\end{pmatrix},
\qquad
\begin{pmatrix}
1&0&0\\
0&0&1\\
0&1&0
\end{pmatrix}.
\]
\item $\mathfrak g\cong\mathfrak g(c)$ with $c<1$ and $c\neq0$, and $g$ is equivalent to the metric 
\[
\begin{pmatrix}
1&0&0\\
0&0&1\\
0&1&0
\end{pmatrix}.
\]
\end{enumerate}
\end{lemma}
\begin{proof}
Let $\{v_1,v_2,v_3\}$ be a basis of $\mathfrak g$ with $\mathfrak h=\mathrm{Span}_{\mathbb R}\{v_1,v_2\}$ and 
\[D
=
\mathrm{ad}(v_3)|_{\mathfrak h}=
\begin{pmatrix}
d_{11}&d_{12}\\
d_{21}&d_{22}
\end{pmatrix}
\] 
for some $d_{ij}\in\mathbb R$. 
Note that $\mathrm{tr}(D)=d_{11}+d_{22}\neq0$ since $\mathfrak g$ is non-unimodular. 
The metric $g$ in this basis takes the form $g=v^1\otimes v^1+v^2\odot v^3$. 
Let 
\[e_1=v_1, \qquad e_2=\frac{\sqrt{2}}{2}(v_2+v_3), \qquad e_3=\frac{\sqrt{2}}{2}(v_2-v_3).\]
Then $\{e_1,e_2,e_3\}$ is an orthonormal basis of $(\mathfrak g,g)$ with $\varepsilon_1=\varepsilon_2=-\varepsilon_3=1$. 
Let $\psi\in\Sigma$ be a left-invariant spinor. In this situation, by Proposition \ref{prop:Dirac_almost_Abelian_isotropic} we have
\[
\slashed{D}\psi
=\frac{1}{4}\begin{pmatrix}
\sqrt{2}(d_{11} + d_{22})-id_{12}&-\sqrt{2}(d_{11} + d_{22})\\
\sqrt{2}(d_{11} + d_{22})&-\sqrt{2}(d_{11} + d_{22})-id_{12}
\end{pmatrix}\psi.
\]
Left-invariant harmonic spinors exist when $\det(\slashed{D})=-\frac{1}{16}d_{12}^2=0$, so $d_{12}=0$. 
A computation shows that in this situation we have 
\[
B_{\mathfrak g}=
\begin{pmatrix}
0&0&0\\
0&0&0\\
0&0&d_{11}^2+d_{22}^2
\end{pmatrix},
\quad
\mathrm{Ric}(g)
=
\begin{pmatrix}
0&0&0\\
0&0&0\\
0&0&d_{11}(d_{22}-d_{11})
\end{pmatrix},
\quad
\mathrm{Scal}(g)=0.
\]
We consider the following cases.

\textbf{Case $d_{11}-d_{22}=0$.} 
In this case the endomorphism $D$ has two equal non-zero eigenvalues, and it is diagonalizable if and only if $d_{21}=0$.  
\begin{enumerate}
\item If $d_{21}=0$, then the Lie algebra $\mathfrak g$ is isomorphic to $\mathbb R^2\rtimes_{\mathrm{Id}}\mathbb R$. 
Moreover, the metric $g$ is flat, thus it is equivalent to the metric 
\[
\begin{pmatrix}
1&0&0\\
0&0&1\\
0&1&0
\end{pmatrix}.
\]
\item If $d_{21}\neq0$, then the Lie algebra $\mathfrak g$ is isomorphic to $\mathfrak g(1)$, $g$ is flat and equivalent to
\[
\begin{pmatrix}
0&0&1\\
0&\mu&0\\
1&0&0
\end{pmatrix},
\qquad\mu>0.
\]
\end{enumerate}

\textbf{Case $d_{11}=0$.} 
Here the endomorphism $D$ has one zero and one non-zero eigenvalue, the Lie algebra $\mathfrak g$ is isomorphic to $\mathfrak g(0)$, $g$ is flat and equivalent to 
\[
\begin{pmatrix}
0&0&1\\
0&1&0\\
1&0&0
\end{pmatrix}.
\]
    
\textbf{Case $d_{22}=0$.} 
Here the endomorphism $D$ has one zero and one non-zero eigenvalue. 
Hence, the Lie algebra $\mathfrak g$ is isomorphic to $\mathfrak g(0)$. 
However, in this case the metric $g$ is non-flat, although it is scalar-flat and the Ricci endomorphism has rank one. 
Hence, $g$ is equivalent to the metric 
\[
\begin{pmatrix}
1&0&0\\
0&0&1\\
0&1&0
\end{pmatrix}.
\]
    
\textbf{Case $d_{11}(d_{22}-d_{11})\neq0$ and $d_{22}\neq0$.} 
In this case the endomorphism $D$ has two non-zero distinct real eigenvalues. 
Hence the Lie algebra $\mathfrak g$ is isomorphic to $\mathfrak g(c)$ with $c<1$ and $c\neq0$. 
We next show that $g$ is equivalent to the metric 
\[
\begin{pmatrix}
1&0&0\\
0&0&1\\
0&1&0
\end{pmatrix}.
\]
Indeed, consider the basis $\{x,y,z\}$ given by 
\[x=(d_{11}-d_{22})v_1+d_{21}v_2,\qquad y=v_2,\qquad z=\frac{2}{d_{11}+d_{22}}v_3.\]
The basis $\{x,y,z\}$ satisfies $[x,y]=0$, $[z,x]=(1+w)x$, and $[z,y]=(1-w)y$, where $w=(d_{11}-d_{22})/(d_{11}+d_{22})$, thus $c=1-w^2=4d_{11}d_{22}/(d_{11}+d_{22})^2<1$.
The matrix of the metric $g$ in this basis is 
\[g=
\frac{1}{d_{11}+d_{22}}
\begin{pmatrix}
(d_{11}+d_{22})(d_{11}-d_{22})^2&0&2d_{21}\\
0&0&2\\
2d_{21}&2&0
\end{pmatrix}.
\]
Taking the automorphism $A\in\mathrm{Aut}(\mathfrak g(c))$ given by
\[
A=\begin{pmatrix}
1/(d_{11}-d_{22})&0&-2d_{21}/(d_{11}+d_{22})(d_{11}-d_{22})^2\\
0&(d_{11}+d_{22})/2&d_{21}^2/(d_{11}+d_{22})(d_{11}-d_{22})^2\\
0&0&1
\end{pmatrix}
\]
we conclude that $A^TgA$ has the claimed form.
\end{proof}

Combining all the previous lemmas we obtain the following result.

\begin{theorem}
\label{thm:3d_nonuni_Lorentzian_classification}
Let $(\mathfrak g,g)$ be a non-unimodular three-dimensional Lorentzian Lie algebra admitting left-invariant harmonic spinors. 
Then $(\mathfrak g,g)$ is one of the cases in \emph{Table \ref{table:Lorentzian_nonuni}}.
\end{theorem}

As as consequence of Theorem \ref{thm:3d_uni_Lorentzian_classification} and Theorem \ref{thm:3d_nonuni_Lorentzian_classification} we get the following fact.
\begin{corollary}
\label{cor:three-dim-lorentz-harmonic-spinors}
All three-dimensional Lie algebras admit at least one Lorentzian metric carrying left-invariant harmonic spinors.
\end{corollary}
Classifications of indefinite metrics on Lie algebras in arbitrary dimension are out of reach in general. 
However, by looking at the classification results of \cite{boucetta-chakkar,ha-lee3} in dimension three, one can conjecture that for a given Lie algebra $\mathfrak g$, there are a number of inequivalent indefinite metrics on $\mathfrak g$. 
Therefore, inspired by Corollary~\ref{cor:three-dim-lorentz-harmonic-spinors}, we state the following general question, which we believe is open.
\begin{question}
Does every Lie algebra admit at least one indefinite metric carrying left-invariant harmonic spinors?
\end{question}

\section*{Appendix}
\label{appendix:unimodular_Lorentzian}

Here we collect useful classification results from \cite{boucetta-chakkar}. 
We describe all Lorentzian metrics on each unimodular three-dimensional Lie algebra, and the signature of the Ricci tensor of each of these metrics.
We summarise the results in Tables \ref{table:simple-algebras-lorentzian-metrics}--\ref{table:non-simple-algebras-lorentzian-metrics}.
\begin{table}
\centering
\caption{Lorentzian metrics on three-dimensional simple Lie algebras}
\footnotesize
\begin{tabular}{p{10mm}p{20mm}p{45mm}M{30mm}} \toprule
$\mathfrak g$ & \textbf{Lie brackets} & \textbf{Metrics (up to equivalence)} & $\sigma(\mathrm{Ric}(g))$ \\ \midrule
$\mathfrak{su}(2)$ & $
\begin{aligned}
[x,y]&=2z \\
[z,x]&=2y \\
[z,y]&=-2x
\end{aligned}
$ & $(\textup{su})=\begin{pmatrix}
    \mu_1&0&0\\
    0&\mu_2&0\\
    0&0&-\mu_3
\end{pmatrix}$ \newline \vskip0.03cm with $\mu_1\geq \mu_2 > 0$, $\mu_3>0$ & 
$(3,0,0), \mbox{ } \mu_1<\mu_2+\mu_3$
$(1,2,0), \mbox{ } \mu_1>\mu_2+\mu_3$ 
$(1,0,2), \mbox{ } \mu_1=\mu_2+\mu_3$
\\ \cmidrule{1-4}
$\mathfrak{sl}(2,\mathbb R)$ & 
$
\begin{aligned}
[x,y]&=2z \\
[z,x]&=2y \\
[z,y]&=2x
\end{aligned}
$ & 
$(\textup{sll1})=\begin{pmatrix} -\mu_1&0&0\\0&\mu_2&0\\0&0&\mu_3\end{pmatrix}$ \newline \vskip0.03cm with $\mu_1>0$, $\mu_2\geq \mu_3>0$ & 
$(3,0,0), \mbox{ } \mu_3<\mu_1-\mu_2$ 
$(1,2,0), \mbox{ } \mu_3>\mu_1-\mu_2$ 
$(1,0,2), \mbox{ } \mu_3=\mu_1-\mu_2$
$(0,1,2), \mbox{ } \mu_3=\mu_2-\mu_1$ 
\\ \cmidrule{3-4}
& & 
$(\textup{sll2})=\begin{pmatrix} \mu_1&0&0\\0&-\mu_2&0\\0&0&\mu_3\end{pmatrix}$ \newline \vskip0.03cm with $\mu_1, \mu_2, \mu_3>0$ & 
$(3,0,0), \mbox{ } \mu_1<\mu_2-\mu_3$
$(1,2,0), \mbox{ } \mu_1>\mu_2-\mu_3$ 
$(1,0,2), \mbox{ } \mu_1=\mu_2-\mu_3$
$(0,1,2), \mbox{ } \mu_1=\mu_2+\mu_3$ 
\\ \cmidrule{3-4}
& & 
$(\textup{sll3})=K\begin{pmatrix} M&\beta&0\\ \beta&N&0\\0&0&a^2\alpha/N\end{pmatrix}$ \newline \vskip0.03cm with $a \neq 0$, $\alpha>0$, $\beta>0$, \newline $K=4/a^2\alpha N$, $M=(\beta^2-\alpha^2)/N$, \newline $N=\sqrt{\alpha^2+\beta^2}$ & 
$(1,2,0), \mbox{ } a^2 \neq 2\alpha $ 
$(0,1,2), \mbox{ } a^2=2\alpha $
\\ \cmidrule{3-4}
& & 
$(\textup{sll4})=K\begin{pmatrix} -N&0&\beta\\ 0&a^2\alpha/N&0\\\beta&0&M\end{pmatrix}$ \newline \vskip 0.03cm with $a \neq 0$, $\alpha<0$, $\beta>0$, \newline $K=4/a^2\alpha N$, $M=(\beta^2-\alpha^2)/N$, \newline $N=\sqrt{\alpha^2+\beta^2}$ & 
$(1,2,0)$
\\ \cmidrule{3-4}
& & 
$(\textup{sll5})=K\begin{pmatrix} u&0&v \\ 0&M&0\\ v&0&u\end{pmatrix}$ \newline \vskip0.03cm with $|u|<v$, $v>0$, \newline $K = 16/(v^2-u^2)$, $M=2(u+v)$ & 
$(1,2,0)$
\\ \cmidrule{3-4}
& & 
$(\textup{sll6})=K\begin{pmatrix} M&-a&0\\-a&N&0\\0&0&8a/b\end{pmatrix}$ \newline \vskip0.03cm with $a,b \neq 0$, $K=1/2ab$, \newline $a=M+8=N-8$ & 
$(1,2,0), \mbox{ } a\neq 2b$ 
$(0,1,2), \mbox{ } a=2b$
\\ \cmidrule{3-4}
& & 
$(\textup{sll7})=K\begin{pmatrix} M&N&0\\ N&S&R\\0&R&4a^4\end{pmatrix}$ \newline \vskip0.03cm with $a>0$, $K=2/a^4(1+2a^2)$, $M=1-4a^4$, $N=(1+2a^2)^{3/2}$, $S=4a^4+6a^2+1$, $R=2a^3\sqrt 2$ & 
$(1,2,0)$ \\
\bottomrule
\end{tabular}
\label{table:simple-algebras-lorentzian-metrics}
\end{table}

\begin{table}
\centering
\caption{Lorentzian metrics on three-dimensional unimodular non-simple Lie algebras}
\footnotesize
\begin{tabular}{p{10mm}p{20mm}p{45mm}M{30mm}} \toprule
$\mathfrak g$ & \textbf{Lie brackets} & \textbf{Metrics (up to equivalence)} & $\sigma(\mathrm{Ric}(g))$ \\ \midrule
$\mathfrak{e}(2)$ & $
\begin{aligned}
[x,y]&=z \\
[z,x]&=y \\
[z,y]&=0
\end{aligned}
$ & 
$(\textup{ee1})=\begin{pmatrix}
    0&1&0\\
    1&u&0\\
    0&0&v
\end{pmatrix}$ 
\newline with $u\geq v > 0$ & 
$(0,0,3), \mbox{ } u=v$ 
$(1,2,0), \mbox{ } u \neq v$
\\ \cmidrule{3-4}
& & 
$(\textup{ee2})=\begin{pmatrix} 0&-1&0\\-1&u&0\\0&0&v\end{pmatrix}$ \newline  with $u<0$, $v>0$ & 
$(3,0,0), \mbox{ } u<-v$ 
$(1,2,0), \mbox{ } u>-v$ 
$(1,0,2), \mbox{ } u=-v$ 
\\ \cmidrule{3-4}
& & 
$(\textup{ee3})=\begin{pmatrix} 0&1&0\\1&0&0\\0&0&u\end{pmatrix}$ \newline  with $u>0$ & 
$(1,2,0)$ 
\\ \cmidrule{1-4}
$\mathfrak{e}(1,1)$ & 
$
\begin{aligned}
[x,y]&=y \\
[z,x]&=z \\
[z,y]&=0
\end{aligned}
$
& 
$(\textup{sol1})=\begin{pmatrix}
    \frac{4}{u^2-v^2}&0&0\\
    0&1&\frac{u}{v}\\
    0&\frac{u}{v}&1
\end{pmatrix}$ 
\newline  with $|u|<v$, $v>0$ & 
$(1,2,0), \mbox{ } u\neq 0$
$(0,1,2), \mbox{ } u=0$
\\ \cmidrule{3-4}
& & 
$(\textup{sol2})=\begin{pmatrix}
    \frac{4}{u^2-v^2}&0&0\\
    0&\frac{u}{v}&-1\\
    0&-1&\frac{u}{v}
\end{pmatrix}$ 
\newline  with $|u|<v$, $v>0$ & 
$(1,2,0), \mbox{ } u>0$ 
$(3,0,0), \mbox{ } u<0$
$(0,0,3), \mbox{ } u=0$
\\ \cmidrule{3-4}
& & 
$(\textup{sol3})=\begin{pmatrix}
    1/(u+v)&0&0\\
    0&-v/u&1\\
    0&1&1
\end{pmatrix}$ 
\newline with $u>0$, $v>0$ & 
$(1,2,0)$ 
\\ \cmidrule{3-4}
& & 
$(\textup{sol4})=\begin{pmatrix}
   1/u&0&0\\
    0&-1&0\\
    0&0&1
\end{pmatrix}$ 
\newline   with $u>0$ & 
$(0,1,2)$
\\ \cmidrule{3-4}
& & 
$(\textup{sol5})=\begin{pmatrix}
    0&0&-2/u\\
    0&1&1\\
    -2/u&1&1
\end{pmatrix}$ 
\newline with $u>0$ & 
$(1,2,0)$
\\ \cmidrule{3-4}
& & 
$(\textup{sol6})=\begin{pmatrix}
    u^2&0&0\\
    0&u&1\\
    0&1&0
\end{pmatrix}$ 
\newline  with $u\neq 0$ & 
$(0,1,2), \mbox{ } u>0$ 
$(1,0,2), \mbox{ } u<0$
\\ \cmidrule{3-4}
& & 
$(\textup{sol7})=\begin{pmatrix}
    0&0&1\\
    0&1&0\\
    1&0&0
\end{pmatrix}$ 
& 
$(0,1,2)$
\\ \cmidrule{1-4}
$\mathfrak{heis}_3$ &  
$
\begin{aligned}
[x,y]&=z \\
[z,x]&=0 \\
[z,y]&=0
\end{aligned}
$
& 
$(\textup{nil+})=\begin{pmatrix}
    1&0&0\\
    0&-1&0\\
    0&0&\lambda
\end{pmatrix}$ 
\newline   with $\lambda > 0$ & 
$(1,2,0)$ 
\\ \cmidrule{3-4}
& & 
$(\textup{nil-})=\begin{pmatrix}
    1&0&0\\
    0&1&0\\
    0&0&-\lambda
\end{pmatrix}$ 
\newline  with $\lambda > 0$ & 
$(3,0,0)$ 
\\ \cmidrule{3-4}
& & 
$(\textup{nil0})=\begin{pmatrix}
    1&0&0\\
    0&0&1\\
    0&1&0
\end{pmatrix}$ &
$(0,0,3)$ \\
\bottomrule
\end{tabular}
\label{table:non-simple-algebras-lorentzian-metrics}
\end{table}


\clearpage

\begin{thebibliography}{100}

\bibitem{ammann-bar}
B.\ Ammann, C.\ B\"ar, \emph{The Dirac operator on nilmanifolds and collapsing circle bundles}. Ann.\ Glob.\ Anal.\ Geom.\ \textbf{16} (1998), no.\ 3, 221--253, \url{https://doi.org/10.1023/A:1006553302362}.

\bibitem{andrada-barberis-dotti-ovando}
A.\ Andrada, M.\ L.\ Barberis, I.\ G.\ Dotti, and G.\ P.\ Ovando, \emph{Product structures on four-dimensional solvable Lie algebras}. Homology Homotopy Appl.\ \textbf{7} (2005), no.\ 1, 9--37, \url{https://doi.org/10.4310/HHA.2005.v7.n1.a2}.

\bibitem{bfgk}
H.\ Baum, T.\ Friedrich, R.\ Grunewald, and I.\ Kath, Twistors and Killing spinors on Riemannian manifolds. B.\ G.\ Teubner Verlagsgesellschaft mbH, Stuttgart, 1991.

\bibitem{bazzoni-merchan-munoz}
G.\ Bazzoni, L.\ Mart\'in-Merch\'an, and V.\ Mu\~noz, \emph{Spin-harmonic structures and nilmanifolds}. Comm.\ Anal.\ Geom.\ \textbf{32} (2024), no.\ 1, 153--201, \url{https://doi.org/10.4310/CAG.240905215149}.

\bibitem{bianchi}
L.\ Bianchi, \emph{Sugli spazi a tre dimensioni che ammettono un gruppo continuo di movimenti}. Memorie di Matematica e di Fisica della Società Italiana delle Scienze, Serie Terza, Tomo XI (1898), 267--352.

\bibitem{bianchi-essay}
L.\ Bianchi, \emph{On the three-dimensional spaces which admit a continuous group of motions}. Gen.\ Relativ.\ Gravit.\ \textbf{33} (2001), no.\ 12, 2171--2253, \url{https://doi.org/10.1023/A:1015357132699}.

\bibitem{boucetta-chakkar}
M.\ Boucetta, A.\ Chakkar, \emph{The moduli spaces of Lorentzian left-invariant metrics on three-dimensional unimodular simply connected Lie groups}. J.\ Korean Math.\ Soc.\ \textbf{59} (2022), no.\ 4, 651--684, \url{https://doi.org/10.4134/JKMS.j210460}.

\bibitem{conti-gilgarcia}
D.\ Conti, A.\ Gil-Garc\'ia, \emph{Almost Abelian pseudo-K\"ahler Lie algebras}. To appear in Ann.\ Sc.\ Norm.\ Super.\ Pisa Cl.\ Sci., \url{https://doi.org/10.2422/2036-2145.202507_006}.

\bibitem{conti-rossi}
D.\ Conti, F.\ A.\ Rossi, \emph{Einstein nilpotent Lie groups}. J.\ Pure Appl.\ Algebra \textbf{223} (2019), no.\ 3, 976--997, \url{https://doi.org/10.1016/j.jpaa.2018.05.010}.

\bibitem{harmful}
D.\ Conti, F.\ A.\ Rossi, and R.\ Segnan Dalmasso, \emph{Harmful structures and Killing spinors on unimodular Lie groups}. arXiv preprint: \url{https://arxiv.org/abs/2509.08507}, 2025.

\bibitem{freibert}
M.\ Freibert, \emph{Cocalibrated structures on Lie algebras with a codimension one Abelian ideal}. Ann.\ Glob.\ Anal.\ Geom.\ \textbf{42} (2012), no.\ 4, 537--563, \url{https://doi.org/10.1007/s10455-012-9326-0}.

\bibitem{friedrich}
T.\ Friedrich, Dirac Operators in Riemannian Geometry. Institut f\"ur Mathematik, Humboldt-Universit\"at, Berlin, Germany.

\bibitem{ha-lee1}
K.\ Y.\ Ha, J.\ B.\ Lee, \emph{Left invariant metrics and curvatures on simply connected three-dimensional Lie groups}. Math.\ Nachr.\ \textbf{282} (2009), no.\ 6, 868--898, \url{https://doi.org/10.1002/mana.200610777}.

\bibitem{ha-lee2}
K.\ Y.\ Ha, J.\ B.\ Lee, \emph{Left invariant Lorentzian metrics and curvatures on non-unimodular Lie groups of dimension three}. arXiv preprint: \url{https://arxiv.org/abs/2209.02208}, 2022.

\bibitem{ha-lee3}
K.\ Y.\ Ha, J.\ B.\ Lee, \emph{Left invariant Lorentzian metrics and curvatures on non-unimodular Lie groups of dimension three}. J.\ Korean Math.\ Soc.\ \textbf{60} (2023), no.\ 1, 143--165, \url{https://doi.org/10.4134/JKMS.j220238}.

\bibitem{hitchin}
N.\ Hitchin, \emph{Harmonic spinors}. Adv.\ Math.\ \textbf{14} (1974), no.\ 1, 1--55, \url{https://doi.org/10.1016/0001-8708(74)90021-8}.

\bibitem{milnor}
J.\ Milnor, \emph{Curvatures of left-invariant metrics on Lie groups}. Adv.\ Math.\ \textbf{21} (1976), no.\ 3, 293--329, \url{https://doi.org/10.1016/S0001-8708(76)80002-3}.

\bibitem{witten}
E.\ Witten, \emph{Search for a realistic Kaluza--Klein theory}. Nucl.\ Phys.\ B \textbf{186} (1981), no.\ 3, 412--428, \url{https://doi.org/10.1016/0550-3213(81)90021-3}.

\end{thebibliography}
\end{document}